\documentclass[a4paper,12pt,final]{amsart}
\usepackage{amsaddr}
\usepackage[a4paper,left=24mm,right=24mm,top=34mm,bottom=34mm]{geometry}
\usepackage[utf8]{inputenc}
\usepackage{color}
\usepackage{xr-hyper} 
\usepackage{hyperref}
\usepackage{cleveref}
\usepackage{showlabels}
\usepackage{todonotes}
\hypersetup{
	colorlinks,
	citecolor=red,
	filecolor=blue,
	linkcolor=blue,
	urlcolor=purple
}


\def\subjclass#1{{\renewcommand{\thefootnote}{}%
		\footnote{\emph{Mathematics Subject Classification (2010):} #1}}}
\def\keywords#1{\par\medskip 	\noindent\textbf{Keywords.} #1}

\usepackage[all]{xy}
\usepackage{amsmath,amssymb,amsfonts,amsthm,graphicx}

\usepackage{multirow}
\usepackage{comment}
\theoremstyle{plain}

\tikzset{node distance=2cm, auto}

\newtheorem{thm}{Theorem}[section]

\newtheorem{cor}[thm]{Corollary}
\newtheorem{prop}[thm]{Proposition}
\newtheorem{df}[thm]{Definition}
\newtheorem{lm}[thm]{Lemma}

\newtheorem*{theorem*}{Theorem}
\newtheorem*{aim*}{Aim}
\newtheorem*{initialaim*}{Initial Aim}
\newtheorem*{conj*}{Conjecture}
\newtheorem*{cor*}{Corollary}
\newtheorem*{prop*}{Proposition}
\newtheorem*{df*}{Definition}
\newtheorem*{lm*}{Lemma}
\newtheorem*{example*}{Example}
\newtheorem*{notation*}{Notation}
\newtheorem*{prob*}{Problem}

\theoremstyle{remark}

\newcommand{\N}{\mathbb{N}}
\newcommand{\Z}{\mathbb{Z}}
\newcommand{\Q}{\mathbb{Q}}
\newcommand{\R}{\mathbb{R}}
\newcommand{\C}{\mathbb{C}}

\newcommand{\F}{\mathbb{F}}

\newcommand{\GenGp}[1]{\langle#1\rangle}

\newcommand{\Norm}[2]{{\rm N}_{#1}({#2})}
\newcommand{\SL}{\operatorname{SL}}
\newcommand{\SU}{\operatorname{SU}}
\newcommand{\tA}{\mathrm A}

\newcommand{\calB}{{\mathcal B}}\newcommand{\calV}{{\mathcal V}}

\newcommand{\calG}{\mathcal G}
\newcommand{\calM}{\mathcal M}
\newcommand{\calP}{\mathcal P}

\newcommand{\wtbT}{\widetilde{\bf T}}
\newcommand{\wtbL}{\widetilde\bL}
\newcommand{\Symm}{\mathfrak{S}}
\newcommand{\Bl}{\operatorname{Bl}}
\newcommand{\Out}{\operatorname{Out}}
\newcommand{\PSL}{\operatorname{PSL}}

\newcommand{\GL}{\operatorname{GL}}
\newcommand{\GU}{\operatorname{GU}}
\newcommand{\bl}{\operatorname{bl}}

\newcommand{\diag}{\operatorname{diag}}
\newcommand{\Irr}{\operatorname{Irr}}
\newcommand{\IBr}{\operatorname{IBr}}
\newcommand{\Inn}{\operatorname{Inn}}
\newcommand{\Aut}{\operatorname{Aut}}
\newcommand{\Ind}{\operatorname{Ind}}
\newcommand{\Res}{\operatorname{Res}}

\newcommand{\ol}{\overline}
\newcommand{\ul}{\underline}

\newcommand{\wh}{\widehat}

\newcommand{\bG}{{{\mathbf G}}}
\newcommand{\bH}{{{\mathbf H}}}

\numberwithin{equation}{section}

\newcommand{\Cl}{\operatorname{Cl}}
\newcommand{\bL}{{{\mathbf L}}}
\newcommand{\bM}{{{\mathbf M}}}
\newcommand{\bK}{{{\mathbf K}}}
\newcommand{\Ze}{\operatorname Z}
\newcommand{\FF}{\ensuremath{\mathbb{F}}}
\newcommand{\bS}{{\mathbf S}}
\newcommand{\bT}{{\mathbf T}}
\newcommand{\NNN}{\operatorname{N}}
\newcommand{\cE}{\mathcal{E}}
\newcommand{\RDL}{\operatorname{R}}
\newcommand{\GF}{ {\bG^F}}
\newcommand{\wGF}{ {\wbG^F}}
\newcommand{\wG}{ {\widetilde G }}
\newcommand{\Cent}{\operatorname{C}}
\newcommand{\Zent}{\operatorname{Z}}
\newcommand{\cP}{{\mathcal P}}
\newcommand{\calN}{{\mathcal N}}
\newcommand{\cD}{{\mathcal D}}
\newcommand{\wt}{\widetilde}
\newcommand{\wbG}{{\widetilde \bG}}
\newcommand{\la}{\lambda}
\newcommand{\lp}{{\ell'}}
\newcommand{\BB}{{\mathbb B}}
\newcommand{\wN}{{\wt N}}

\usepackage{cleveref,enumitem}
\newlist{asslist}{enumerate}{1} 
\setlist[asslist]{label=(\roman*), ref=\thethm(\roman*)}

\crefalias{asslistenumi}{Assumption} 

\newlist{thmlist}{enumerate}{1} 
\setlist[thmlist]{label=(\alph*), ref=\thethm(\alph*)}

\title{On the Alperin-McKay conjecture for simple groups of type $\mathrm{A}$}

\begin{document}

\date{\today}
\author{Julian Brough and Britta Sp\"ath}
\thanks{The first author thanks the Humbold Foundation for its support.}
\thanks{This material is partially based upon work supported by the NSF Grant No. DMS-1440140 while the second author was in residence at the MSRI in Berkeley, California, during the Spring 2018 semester. Both authors like to thank the research training group \emph{GRK 2240: Algebro-Geometric Methods in Algebra, Arithmetic and Topology}, funded by the DFG}
\address{School of Mathematics and Natural Sciences University of Wuppertal, Gau\ss str. 20, 42119 Wuppertal, Germany}
	\email{brough@uni-wuppertal.de,\, bspaeth@uni-wuppertal.de} 
\begin{abstract}
	In this paper characters of the normaliser of $d$-split Levi subgroups in ${\SL}_n(q)$ and ${\SU}_n(q)$ are parametrized with a particular focus on the Clifford theory between the Levi subgroup and its normalizer.
	These results are applied to verify the Alperin-McKay conjecture for primes $\ell$ with $\ell\nmid 6(q^2-1)$ and the Alperin weight conjecture for $\ell$-blocks of those quasi-simple groups with abelian defect. The inductive Alperin-McKay condition and inductive Alperin weight condition by the second author are verified for certain blocks of ${\SL}_n(q)$ and ${\SU}_n(q)$.

	\keywords{Alperin-McKay conjecture, inductive Alperin-McKay condition,
		inductive blockwise Alperin weight condition, special linear group, special unitary group}
\end{abstract}

\date{\today}
\subjclass{ 20C20 (20C33 20C34)}
\maketitle

\excludecomment{commentMe}

\vspace{-10mm}
\normalsize

\section{Introduction}
The McKay conjecture and its blockwise version, the Alperin-McKay conjecture, are two fundamental conjectures in representation theory of finite groups. 

For a finite group $G$, a prime $\ell$ and $N$ the normalizer of a Sylow $\ell$-subgroup of $G$, the McKay conjecture is concerned with the cardinality of $\Irr_{\ell '}(X)$ where $\Irr_{\ell'}(X)=\{ \chi\in \Irr(X)\mid \chi(1)_{\ell}=1\}$ and states that 
$$|\Irr_{\ell '}(G)|=|\Irr_{\ell '}(N)|.$$ 

For an $\ell$-block $B$ of $G$ with defect group $D$ and $b$ the Brauer correspondent of $B$, an $\ell$-block of the normalizer in $G$ of $D$, the Alperin-McKay conjecture claims that $$|\Irr_0(B)|=|\Irr_0(b)|,$$ where $\Irr_0(C)$ is the set of height $0$ characters in a block $C
$. By summing over blocks with maximal defect, the Alperin-McKay conjecture implies the McKay conjecture.

Alperin introduced with the Alperin Weight Conjecture a similar global-local conjecture that claimed that the number of irreducible Brauer characters of a given $\ell$-block coincides with the number of weights of this $\ell$-block. If the block $B$ has abelian defect group, the conjecture is equivalent to \[|\IBr(B)|=|\IBr(b)|,\]
where $b$ is as before the Brauer correspondent of $B$. 
A more structural explanation of both conjectures was suggested by Brou\'e, namely that in case of abelian defect the blocks $B$ and $b$ are derived equivalent. 

In \cite{IMNRedMcKay,AMSp} the McKay and the Alperin-McKay conjecture have been reduced to finite simple groups. It remains to check the corresponding so-called {\it inductive condition} for all finite simple groups and primes $\ell$. 
Substantial progress has been made towards proving the McKay conjecture along that line.
In particular, using the inductive McKay condition, Malle and the second author established the McKay conjecture for $\ell=2$ \cite{MalSpOddDegree}.
Furthermore for odd primes the inductive McKay condition has been verified in multiple cases of simple groups \cite{SpIMDefChar, CabSpMZ, CabSpIMTypeC, CabSpCharTypeA}.
Less cases have been checked for the inductive Alperin-McKay (AM) condition. \cite{AMSp} verified the inductive AM condition for simple groups of Lie type, when $\ell$ is the defining characteristic, and for alternating groups, when the prime $\ell$ is odd. Additionally \cite{BAWMa, ASF_Sp6} have dealt with simple groups of types $^2\mathrm{B}_2, {}^2\mathrm{G}_2$ and $^2\mathrm{F}_4$. Further results have been established in \cite{CabSpAMTypeA, KosSpAMBAWCy, KosSpAM2BlCy} by considering particular structures of the defect group of the block.

The present paper is concerned with the Alperin-McKay condition for quasi-simple groups of type $\tA$ and primes $\ell$ different from the defining characteristic with $\ell\geq 5$. Note that the inductive Alperin-McKay (AM) and the blockwise Alperin weight (BAW) conditions hold for most blocks in the defining characteristic according to \cite{AMSp} and \cite{BAWSp}. 
In order to verify the inductive Alperin-McKay condition a new criterion is introduced in Theorem~\ref{NewIndAmCond}, which will have applications to other series of simple groups, see \cite{CabASFSp}. It complements the criterion given in \cite{CabSpAMTypeA}.
This leads to the following statement, where we write $\SL_n(-q)$ for $\SU_n(q)$ and $\GL_n(-q)$ for $\GU_n(q)$.

\begin{thm}\label{IndAMAbBl}\label{thmA}
Let $\ell$ be a prime, $q$ a prime power and $\epsilon\in \{\pm 1\}$ with $\ell\nmid 3q( q-\epsilon)$, 
$\bG:=\SL_n(\overline\FF_q)$, $G:=\SL_n(\epsilon q)$, $B_0$ an $\ell$-block of $G$ with defect group $D$, and $B$ the $\GL_n(\epsilon q)$-orbit containing $B_0$. Assume that $\PSL_n(\epsilon q)$ is simple, $G$ is its universal covering group and the stabilizer $\Out(G)_B$ is abelian. 
\begin{thmlist}
	\item \label{thm11a}
The inductive AM condition from Definition 7.2 of \cite{AMSp} holds for $B_0$.
\item \label{thm11b} Let $d$ be the order of $q$ in $(\Z/\ell\Z)^\times$. 
If $D$ is abelian and $\Cent_{\bG}(D)$ is a $d$-split Levi subgroup of $\bG$, then the inductive BAW condition from \cite{BAWSp} holds for $B_0$.
\end{thmlist}
\end{thm}

In Corollary \ref{corindAM}, the inductive AM condition is also proven for blocks that are $\GL_n(\epsilon q)$-stable and satisfy another similar assumption. 

For unipotent blocks Feng has established in \cite{Z_Feng} the inductive BAW condition for unipotent blocks and Li-Zhang have treated
in \cite{LiZhang} other blocks under additional assumptions on the outer automorphism group. Also Li constructed in \cite{Li_Sp2nq} an equivariant bijection for the inductive BAW condition in symplectic groups under some assumption on the $\ell$-modular decomposition matrix. More particular cases of simple groups of small rank were checked in \cite{FengLiLi,Schulte,ASF_Sp6}.

In the proof of Theorem \ref{thmA} a main step is to parametrize the characters of the normalizers of $d$-split Levi subgroups. These normalizers serve as local subgroups in the inductive AM condition, see Theorem~\ref{71b}. 
We investigate the action of automorphisms on those characters in terms of their parameters.
Essential is to understand the Clifford theory of irreducible characters of a $d$-split Levi subgroup $L$ in $\NNN_{G}(L)$.
Furthermore, we consider the action of the stabilizer ${\Aut}(G)_{B,L}$ on the irreducible characters and verify that the corresponding inertia groups are of a particular structure. 

\begin{thm}\label{StarConddSpLevi}
Let $\bG:=\SL_n(\overline \F_q)$, $\wt\bG:=\GL_n(\overline \F_q)$ , 
		$F:\wbG\rightarrow \wbG$ a Frobenius endomorphism defining an $\FF_q$-structure, $\bL$ a $d$-split Levi subgroup of $({\bG}, F)$, $N_0:=\Norm{\GF}{{\bf L}}$ and $\wt N_0:=\Norm{{\wt\bG^F}}{\bL}$.
\begin{thmlist}
\item Every $\lambda\in\Irr(\bL^F)$ extends to its inertia group in $N_{0}$.
\item \label{thm12b}
 Let $E_0\leq\Aut(\wt\bG^F)$ be the image of $E$ defined in 3.A and let $\psi\in \Irr ( \NNN_0)$. Then there exists a $\wt \NNN_0$-conjugate $\psi_0$ of $\psi$ such that 
\begin{asslist}
\item $O_0=(\wt {\bG}^F\cap O_0)\rtimes (E_0\cap O_0)$ for $O_0:={\bG}^F(\wt {\bG}^F\rtimes E_0)_{\bL,\psi_0}$, and
\item $\psi_0$ extends to $({\bG}^F\rtimes E_0)_{\bL,\psi_0}$.
\end{asslist}
\end{thmlist}
\end{thm}

For groups of Lie type with abelian Sylow $\ell$-subgroup, bijections implying the Alperin-McKay conjecture and blockwise Alperin weight were constructed in \cite[Theorem 2.9]{BAWMa} assuming the above \ref{thm12b}(a) for analogous local subgroups.
As a consequence of our proof of Theorem~\ref{StarConddSpLevi}, Alperin-McKay conjecture holds via \cite[Theorem 2.9 and Corollary 3.7]{BAWMa} for special linear and unitary groups with abelian Sylow $\ell$-subgroup.
Thanks to Theorem~\ref{Conjunip} we are able to generalize Malle's approach from \cite{MalleHeight0}, where he constructed a bijection for the inductive McKay condition. 
By considerations inspired by \cite[\S 10]{SpExz} we deduce from this the Alperin-McKay conjecture for all blocks, using results of Puig and Zhou on the so-called {\it inertial blocks}. 
Note that for $\ell\mid q$ the Alperin-McKay conjecture was proven in \cite{AMSp} based on earlier work by Green-Lehrer-Lusztig while for $\ell\mid (q-\epsilon)$ results of Puig in \cite[\S 5]{Puig_Scope} imply the conjecture for most $\ell$-blocks of $\SL_n(\epsilon q)$ with abelian defect.

\begin{thm}\label{AMTypeA}
Let $G =\SL_n(\epsilon q)$. Let $\ell$ be a prime with $\ell\nmid 3q(q-\epsilon)$.
\begin{thmlist}
\item \label{thm13a}
The Alperin-McKay Conjecture holds for all $\ell$-blocks of $G$.
\item \label{thm13b}
The Alperin weight Conjecture holds for all $\ell$-blocks of $G$ with abelian defect.
\end{thmlist}
\end{thm}

This paper is organised in the following way: in Section~\ref{RefomIndAM} we provide a criterion for the inductive AM condition and give some helpful statements using inertial blocks.
Then in Section~\ref{dSplitLevi} we construct explicitly the $d$-split Levi subgroups, their normalizers and we highlight some important properties of the irreducible characters.
The Clifford theory between $d$-split Levi subgroups and their normalizers is studied in Section~\ref{CliffThdSpLevi} in order to prove Theorem~\ref{StarConddSpLevi}. 
The final sections deduce from this the main results using $d$-Harish-Chandra theory and Jordan decomposition.

{\bf Acknowledgement.} The authors thank Lucas Ruhstorfer for useful discussions on the subject. Additionally we thank Gunter Malle for many useful comments on an earlier version.

\section{The inductive AM condition}
\label{RefomIndAM}

In this section we recall the inductive AM condition and give a criterion for proving it in our situation. The proof requires some arguments relying on Dade's ramification group. 
We also give a short lemma that will later be used to verify the Alperin-McKay conjecture. 

For characters of finite groups and blocks we will freely make use of the notation from \cite{IsaChTh} and \cite{NavBl}. Our blocks are considered with regard to a prime $\ell$.
For an $\ell$-block $c$ of a subgroup of a finite group $H$ we denote the induced block by $c^H$. For a generalized character $\chi$ of $H$ we denote by $\Irr(\chi)\subseteq\Irr(H)$ its set of irreducible constituents and for an irreducible character $\chi$ we denote by $\bl(\chi)$ the $\ell$-block that it belongs to. We denote by $\chi^\circ$ the restriction of $\chi$ to the $\ell '$-elements of $H$. Recall that whenever the group $A$ acts on a set $M$ we denote by $A_m$ the stabilizer of an element $m$ of $M$ in $A$.

\subsection{The inductive AM condition for a set of blocks}
In \cite[Definition 7.1]{AMSp} the inductive AM condition for a simple group and a prime $\ell$ were introduced. An alternative version, relative to a radical $\ell$-subgroup was given in \cite[Section 7.1]{CabSpMZ}. In addition \cite[Def.~3.2]{KosSpAMBAWCy} provided one related to a single block. 

First recall the following standard notations.
For $D$ an $\ell$-subgroup of a finite group $X$ we denote by $\Bl(X\mid D)$ the set of $\ell$-blocks of $X$ which have $D$ as a defect group.
Recall the notation 
$\Irr_0(B)$ for the set of height zero characters in a block $B$. For any subset $\calB\subset \Bl(X)$ let the sets $\Irr(\calB)$ and $\Irr_0(\calB)$ be analogously defined. 

\begin{df}\label{IndAMCond}
Let $S$ be a finite non-abelian simple group and $\ell$ a prime. 
Let $G$ be the universal covering group of $S$ and $B\in \Bl(G)$ with defect group $D$.
Then we say that {\bf the inductive AM condition holds for $B$} if
\begin{asslist}
\item there exists some $\Aut (G)_{B,D}$-stable subgroup $M$ such that $\NNN_G(D)\leq M\lneq G$.
\item\label{indAmii}
For $B'\in\Bl(M)$ with $B'^G=B$ there exists an $\Aut (G)_{B,D}$-equivariant bijection
\[
\Omega_B:\Irr_0(B) \longrightarrow \Irr_0(B'),
\]
such that
\begin{itemize}[label={\textbullet}]
\item $\Omega_B\left( \Irr_0(B) \cap \Irr (G\mid \nu ) \right)\subseteq \Irr (M\mid \nu)$ for every $\nu\in \Irr (\Zent(G))$,
\item ${\bl}(\chi)={\bl}\left( \Omega_B(\chi)\right) ^G$ for every $\chi\in \Irr_0(B)$.
\end{itemize}
\item For every character $\chi\in\Irr_0(B)$ there exists a group $A$ and characters $\wt\chi$ and $\wt\chi'$ such that 
\begin{enumerate}
\item For $Z:=\ker (\chi)\cap \Zent(G)$ and $\ol{G}:=G/Z$ the group $A$ satisfies $\ol{G}\lhd A$, $\Cent_A(\ol{G})=\Zent(A)$ and $A/\Zent(A)\cong \Aut (G)_{\chi}$,
\item $\wt{\chi}\in \Irr (A)$ is an extension of $\ol{\chi}$, where $\ol{\chi}\in \Irr (\ol{G})$ lifts to $\chi$,
\item for $\ol{M}:=MZ/Z$ and $\ol{D}:=DZ/Z$ the character $\wt{\chi}'\in \Irr\left( \ol{M}\NNN_A(\ol{D}) \right)$ is an extension of $\ol{\chi}'$, where $\ol{\chi}'$ lifts to $\chi':=\Omega_B(\chi)\in \Irr (M)$,
\item $\bl\left( \Res^A_J(\wt{\chi})\right) = \bl\left( \Res^{\ol{M}\NNN_A(\ol{D})}_{\ol{M}\NNN_J(\ol{D})}(\wt{\chi}') \right)^J$ for every $J$ with $\ol{G}\leq J\leq A$,
\item $\Irr (\Res^{A}_{\Zent(A)}( \wt{\chi}))=\Irr (
\Res^{\overline M\NNN_A(\overline D)}_{\Zent(A)}( \wt{\chi}')
)$.
\end{enumerate}
\end{asslist}
\end{df}

In particular, \cite[Proposition 2.2]{CabSpAMTypeA} shows that if the inductive AM condition holds for any $B\in\Bl(G)$, then $S$ satisfies the inductive AM condition \cite[Definition 7.2]{AMSp} with respect to $\ell$.
Furthermore, \cite[Theorem 4.1]{CabSpAMTypeA}, an alternative version of the above condition was given in the case the radical subgroup is a Sylow $\ell$-subgroup.
In other words, a modified version to consider blocks with maximal defect.
The aim is to provide an analogous statement which focuses on blocks with a ``nice'' stabilizer subgroup in the automorphism group of $G$. 

\subsection{Some block theory}
To prove Theorem~\ref{NewIndAmCond} below, Dade's {\it ramification group} provides a fundamental tool. It was introduced in \cite{DadeBlExt} and then reformulated in \cite{MurBlNorSub} by Murai. We use it to study blocks and to define a bijection required for Definition~\ref{IndAMCond}.

For each block $b$ we denote by $\la_b$ the associated map defined in \cite[\S 3]{NavBl}. For a group $G$ and $x\in G$ we denote by $\Cl_G(x)^+$ the sum of the $G$-conjugacy class containing $x$ in the group algebra $\Z[G]$. 
\begin{df}
Let $X\lhd \wt X$ and $B\in \Bl(X)$, then the group $\wt X[b]$ is defined by
\[
\wt X[B]:=\Big\{ x\in \wt X_B \mid \lambda_{b^{(x)}}\big( \Cl_{\langle X,x\rangle }(y)^+\big)\ne 0 \text{ for some } y\in xX\Big\},
\]
where $b^{(x)}$ denotes an arbitrary block of $\langle X,x\rangle$ which covers $B$.
(This definition of $\wt X[B]$ is independent of the choice of $b^{(x)}$, see \cite[3.3]{MurBlNorSub}.)
\end{df}

In the above situation let $\chi\in\Irr (B)$. Then $\wt X[B]\leq \wt X_\chi\leq \wt X_B$
(see \cite[Lemma 3.2]{KosSpCliffTh}).

For the proof of Theorem~\ref{NewIndAmCond} we use the following technical lemma. Recall that for $X\lhd \wt X$ and $B\in\Bl(X)$ we denote by $\Bl(\wt X\mid B)$ the set of all blocks of $\wt X$ covering $B$. 

\begin{lm}\label{CovBlBrCorr}
Let $X\lhd \wt X$ be finite groups with abelian quotient $\wt X/X$, and
$B\in \Bl(X)$ with defect group $D$. Let $M\leq X$ with $M\geq {\NNN}_X(D)$, $\wt M\leq \wt X$ with $\wt M\cap X=M$ and $\wt M\geq M {\rm N}_{\wt X}(D)$, $b\in\Bl(M)$ with $b^X=B$ and $\wt b\in\Bl(\wt M\mid b)$.

Let $\wh M\leq \wt M_b$, $\wh X:=X\wh M$ and $\wt B:=\wt b^{\wt X}$.
\begin{thmlist}
\item If $\widehat b \in\Bl(\wh M \mid b)$ is covered by $\wt b$, 
$\widehat B=\widehat b^{\wh M X}$ covers $B$ and is covered by $\wt B$. 
\item If $\widehat B\in\Bl(\wh X \mid B)$ is covered by $\wt B$, then $\wt b$ covers some $\wh b\in\Bl(\wh M\mid b)$ with $\wh b^{\wh X}=\wh B$. 
\end{thmlist}
\end{lm}
\begin{proof}
	Note that since $\wt X/X$ is abelian, $\widehat M\lhd \wt M$. By \cite[9.28]{NavBl}, $\widehat B$ is defined, covers $B=b^X$, and is covered by $\wt b^{\wt X}=\wt B$. 
	
	Via \cite[9.28]{NavBl} we see that $\widehat c\mapsto \widehat c^{\widehat X}$ defines a bijection $\Bl(\wh M\mid b)\rightarrow \Bl(\wh X\mid B)$. Hence there exists a block $\widehat b$ with $\wh b^{\wh X}=\wh B$. The set $\Bl(\wt M\mid \wh b)$ is contained in $\Bl(\wt M\mid b)$ and is, via block induction, in bijection with $\Bl(\wt X\mid \wh B)$. Clearly $\wt B\in \Bl(\wt X\mid \wh B)$ and hence there exists some $\wt b'\in \Bl(\wt M\mid \wh b)$ with $\wt b'^{\wt X}=\wt B$. Since $\wt b^{\wt X}=\wt B$, this implies $\wt b=\wt b'$ and hence $\wt b$ covers $\wh b$. 
\end{proof}

\subsection{Alternative inductive condition} 

In this section a criterion for the inductive AM condition adapted to simple groups of Lie type is given.
The condition is closely related to the one in \cite[Theorem 4.1]{CabSpAMTypeA}.
In fact conditions $(i)-(iv)$ are the same, and condition $(v)$ which considers the structure of stabilizers of a block is altered.
Naturally, this will limit which blocks these conditions can be considered for.

\begin{thm}\label{NewIndAmCond}
Let $S$ be a finite non-abelian simple group and $\ell$ a prime dividing $|S|$.
Let $G$ be the universal covering group of $S$, $D$ a radical $\ell$-subgroup of $G$ and $\calB\subseteq \Bl (G\mid D)$ a $\wt G_D$-stable subset with $(\wt GE)_B\leq (\wt GE)_\calB$.
Assume we have 
a semi-direct product $\wt G \rtimes E$, 
a $\Aut (G)_{\calB,D}$-stable subgroup $M$ with $\NNN_G(D)\leq M\lneq G$ 
and a group $\wt M\leq \wt G$ with $\wt M\geq M\NNN_{\wt G}(D)$ and $M=\wt M\cap G$ such that the following conditions hold:
\begin{asslist}
\item 
\begin{itemize}[label={\textbullet}]
\item $G=[\wt G,\wt G]$ and $E$ is abelian,
\item $\Cent_{\wt G\rtimes E}(G)=\Zent(\wt G)$ and $\wt GE/\Zent(\wt G)\cong \Inn (G)\Aut (G)_{D}$ by the natural map,
\item any element of $\Irr_0(\calB)$ extends to its stabiliser in $\wt G$,
\item any element of $\Irr_0(\calB')$ extends to its stabiliser in $\wt M$.
\end{itemize}
\item \label{thm24ii}
Let $\calB'\subseteq \Bl (M)$ be the set of all Brauer correspondents of the blocks in $\calB$.
For $\calG:=\Irr\left( \wt G\mid \Irr_0(\calB) \right)$ and $\calM:=\Irr\left( \wt M\mid \Irr_0(\calB') \right)$ there exists an $\NNN_{\wt GE}(D)_\calB$-equivariant bijection
\[\wt{\Omega}:\calG \longrightarrow \calM\]
with
\begin{enumerate}
\item $\wt{\Omega}\left( \calG\cap \Irr(\wt G\mid \wt{\nu}) \right)=\calM\cap \Irr(\wt M\mid \wt{\nu})$ for all $\wt{\nu}\in \Irr\left( \Zent(\wt G) \right)$,
\item $\bl\left( \wt{\Omega}(\wt{\chi}) \right)^{\wt G}=\bl (\wt{\chi})$ for all $\wt{\chi}\in \calG$, and
\item $\wt{\Omega}(\wt{\chi}\wt{\mu})=\wt{\Omega}(\wt{\chi})\Res_{\wt M}^{\wt G}(\wt{\mu})$ for every $\wt{\mu}\in \Irr(\wt G\mid 1_G)$ and every $\wt{\chi}\in \calG$.
\end{enumerate}

\item For every $\wt{\chi}\in \calG$ there exists some $\chi_0\in \Irr (G\mid \wt{\chi})$ such that
\begin{itemize}[label={\textbullet}]
\item $(\wt G\rtimes E)_{\chi_0}=\wt G_{\chi_0}\rtimes E_{\chi_0}$, and 
\item $\chi_0$ extends to $G\rtimes E_{\chi_0}$.
\end{itemize}

\item For every $\wt{\psi}\in \calM$ there exists some $\psi_0\in \Irr (M\mid \wt{\psi})$ such that
\begin{itemize}[label={\textbullet}]
\item $O=(\wt G\cap O)\rtimes (E\cap O)$ for $O:=G(\wt G\times E)_{D,\psi_0}$, and 
\item $\psi_0$ extends to $M(G\rtimes E)_{D,\psi_0}$.
\end{itemize}

\item \label{NewIndAmCondv} For any $\wt G$-orbit $B$ in $ \calB$ the group $\Out (G)_B$ is abelian.
\end{asslist}
Then the inductive AM condition holds for all $\ell$-blocks in $\calB$
\end{thm}

Note that assumption~\ref{NewIndAmCondv} implies that the stabilizer subgroup in $\Out(G)$ of characters of $G$ is stable under $\wt G$-conjugation.
The following statement highlights a situation, based upon the underlying central character, for which this condition of Theorem~\ref{NewIndAmCond} holds.

The above criterion is tailored to simple groups $S$ of Lie type and there are canonical candidates for the groups $G$, $\wt G$ and $E$.
If $S$ is of Lie type $\mathrm{B}, \mathrm{C}$ or $\mathrm{E}_7$, then condition \ref{NewIndAmCondv} holds with the usual choices of the groups. 

Let $\nu\in\Irr(\Ze(G))$ with $\Ze(G)_\ell\leq\ker(\nu)$. If $\Zent(G) \rtimes E_\nu$ is abelian, then $\Out(G)_b$ is abelian for every $\wt G$-orbit $b$ in $\Bl(G\mid \bl(\nu))$.
If the character $\nu$ is faithful, condition~\ref{NewIndAmCondv} holds for every $\wt G$-orbit $b$ in $\Bl(G\mid \bl(\nu))$.

As said before, Theorem~\ref{NewIndAmCond} and \cite[Theorem 4.1]{CabSpAMTypeA} coincide in all but the last assumption. Hence the proof of Theorem~\ref{NewIndAmCond} requires altering the previous proof in all situations where the last assumption is used. We now construct a bijection $\Omega_\calB :\Irr_0(\calB)\rightarrow \Irr_0(\calB')$. Note that the proof \cite[Theorem 4.1]{CabSpAMTypeA} gives in all cases a bijection $\Omega_\calB :\Irr_0(\calB)\rightarrow \Irr_0(\calB')$ but it is not clear if the other requirements are satisfied. 

\subsubsection{A bijection $\Omega_\calB$ for Definition~\ref{IndAMCond}(ii):} 

We first choose in $\calB$ a $(\wt G E)_D$-transversal $\calB_0$. 
The idea is to define $\Omega_\calB$ on a $(\wt G E)_D$-transversal $\calG_0$ in $\Irr_0(\calB)$. 

Let $\chi_0\in \calG_0$. 
Note that $(\wt G E)_{\chi_0}=\wt G_{\chi_0} E_{\chi_0}$ by assumption (iii) and (v). 
Let $\chi\in\Irr(\wt G\mid \chi_0)$ and set $\phi:=\wt \Omega(\chi)$, $B:=\bl(\chi_0)$, $\wt B:=\bl(\chi)\in \Bl(\wt G\mid B)$ and $\wt b=\bl (\phi)$.
Then $\wt b^{\wt G}=\wt B$ by condition \ref{NewIndAmCond}$(ii).2$.
By \cite[9.28]{NavBl} the block $b\in\Bl(M)$ with $b^G=B$ is covered by $\wt b$.

Let $\wh G$ to be the group with $G\leq \wh G\leq \wt G_{B}[B]$ such that $\wh G/G$ is a $\ell'$-Hall subgroup of $\wt G_B[B]/G$. 
As every character of $G$ extends to its stabilizer in $\wt G$, it follows that $\chi_0$ extends to $\wh G\leq \wt G_{B}[B]\leq \wt G_{\chi_0}$.
Moreover by Clifford theory $\Irr(\Res^{\wt G}_{\wh G}(\chi))\cap \Irr(\wh G\mid \chi_0)=\{\wh \chi_0\}$, so set $\wh B:=\bl(\wh \chi_0)$.

According to Lemma \ref{CovBlBrCorr} there exists $\wh b\in\Bl (\wh M \mid b)$ for $\wh M:=\wt M\cap \wh G $ with $\wh b^{\wh G}=\wh B$ that is covered by $\wt b$. Since $\bl(\phi)$ covers $\wh b$ there exists some $\wh \phi_0\in \Irr(\Res^{\wt M}_{\wh M}(\phi))\cap\Irr ( \wh b)$. 
As $\chi_0$ extends to $\wt G$, restriction gives a bijection between $\Irr (\wh b)$ and $\Irr (b)$, see \cite[Lemma 3.7]{KosSpCliffTh}.
Thus $\phi_0:= \Res^{\wh M }_M( \wh\phi_0)\in \Irr (b)$ and we set $\Omega_\calB(\chi_0)=\phi_0$.
\vspace{4mm}

We first ensure that $\Omega_\calB$ satisfies the properties from \ref{indAmii}:
In order to prove $$\Omega_\calB\left( \Irr_0(\calB) \cap \Irr(G\mid \nu ) \right)\subseteq \Irr(M\mid \nu)$$ for every $\nu\in \Irr(\Zent(G))$ it is sufficient to check that for $\chi_0$ as above $\Omega_\calB(\chi_0) \in \Irr(M\mid \nu)$ where $\nu\in\Irr(\Res^G_{\Zent(G)} \chi_0)$. 
Note that $\wt{\Omega}(\chi)\in\Irr(\wt M\mid \wt \nu)$ for some $\wt \nu\in \Irr(\Zent(\wt G)\mid \nu)$, where $\chi\in \Irr(\wt{G}\mid \nu)$. This implies $\Omega_\calB(\chi_0)=\phi_0 \in \Irr(M\mid \nu)$.

To prove the second half of \ref{indAmii} it is sufficient to check ${\bl}(\chi)={\bl}\left( \Omega_\calB(\chi)\right) ^G$. 
By the construction of $\Omega_\calB$, the block $\bl(\chi_0)$ is covered by $\wh B$ and $\bl(\phi_0)$ is covered by $\wh b$, where $\phi_0=\Omega_\calB(\chi_0)$ and $\wh b^{\wh G}=\wh B$. Let $b\in\Bl(M)$ with $b^G=B$. Then by construction, both $b$ and $\bl (\phi_0)$ are covered by $\wh b$. By definition of $\wh M$ this implies $b=\bl(\phi_0)$, since
$\wh M\leq M[b]$ as $\wt G[B]=G\wt M[b]$.

For the proof of Theorem~\ref{NewIndAmCond}, it remains to show that any $\chi_0\in \Irr_0(\calB)$ satisfies the condition given by Definition~\ref{IndAMCond}(iii).
By assumption~\ref{NewIndAmCond}(v), we can focus on $\chi_0\in \calG_0$.
Furthermore, by the proof of \cite[Theorem 4.1]{CabSpAMTypeA}, it is enough to verify the following proposition, which is an analogue of \cite[Proposition 4.2]{CabSpAMTypeA}.
 \begin{prop}
Let $\chi\in\calG_0$ and $\Omega_\calB$ be the bijection constructed above.
Then there exists a group $A$ and characters $\wt{\chi}$ and $\wt{\phi}$ such that
\begin{asslist}
 \item for $Z:=\ker(\chi)\cap \Zent(G)$ and $\ol G:=G/Z$, then $\ol G\lhd A$, $\Cent_A(\ol G)=\Zent(A)$ and $A/\Zent(A)\cong \Aut (G)_\chi$,
 \item $\wt{\chi}\in\Irr(A)$ is an extension of the character $\ol{\chi}\in\Irr(\ol G)$ that lifts to $\chi$,
 \item $\wt{\phi}\in \Irr\left(\ol M\NNN_A(\ol D)\right)$ is an extension of the character $\ol{\phi}\in \Irr(\ol M)$ that lifts to $\phi:=\Omega(\chi)$ where $\ol M:=M/Z$ and $\ol D:=DZ/Z$,
 \item there exists a group $J$ with $\ol G\Zent(A)\leq J\lhd A$ with abelian quotients $A/J$ and $J/\ol G$, such that
 \[
\bl\left(\Res^A_{J_2}(\wt{\chi})\right) = \bl\left(\Res^{\ol M\NNN_A(\ol D)}_{\ol M\NNN_{J_2}(\ol D)}(\wt{\phi})\right)^{J_2} 
\text{ for every $J_2$ with $G\leq J_2\leq J$}
 \]
 \item $\Irr (\Res^A_{ \Zent(A)}(\wt{\chi}))=\Irr (
 \Res^{\overline M \NNN_A(\overline D)}_{\Zent(A)}(\wt{\phi}))$. 
\end{asslist} 
\end{prop}
 
\begin{proof} Let $\chi_1\in \Irr(\wt G_{\chi_0}\mid \chi_0)$ and $\phi_1 \in 
\Irr(\wt M_{\phi_0}\mid \phi_0)$ be the Clifford correspondents to $\chi$ and $\phi$ respectively.
They satisfy $\Res^{\wt G_{\chi_0}}_{\wh G}( \chi_1)= \wh \chi_0$ and
$\Res^{\wt M_{\phi_0}}_{\wh M}( \phi_1)= \wh \phi_0$ and hence 
$\bl(\Res^{\wt G_{\chi_0}}_{\wh G}(\chi_1))=
\bl (\Res^{\wt M_{\phi_0}}_{\wh M}(\phi_1))^{\wh G}$. 
In addition $\wt G[B]=G\wt M[b]$ by \cite{MurBlNorSub} and thus 
$\wh G=\wh G\cap G\wt M[b]=G(\wh G\cap \wt M[b])=G\wh M$.
Note that $M=G\cap \wh M$ and $\ell \nmid | \wh G:G|$. Hence $
\bl ( 
\Res^{\wh G_{}}_{{\GenGp{G,x}}}(\wh\chi_0)
)=
\bl ( \Res^{\wh M_{ }}_{\GenGp{M,x}} (\wh\phi_0))^{\GenGp{G,x}} $
for every $x\in \wh M$, see \cite[Lemma 2.4]{KosSpCliffTh}.

Let $G\leq J_1\leq \wt G[B]$, then $J_1=G(J_1\cap \wt M[b])$.
Since $(J_1\cap \wt M[b])_{\ell '}\leq \wh M$ 
we see $ \bl ( \Res^{\wt G_{\chi_0}}_{J_1}(\chi_1)
)=\bl (\Res^{\wt M_{\phi_0}}_{{J_1\cap \wh M}} (\phi_1))^{J_1}$ 
according to \cite[Lemma 2.5]{KosSpCliffTh}.

Let $G\leq J_2\leq \wt G_{\chi_0}$ and $J_1:=J_2\cap \wt G[B]=J_2[B]$. Then $
\bl (\Res^{\wt G_{\chi_0}}_{J_1}(\chi_1))=
\bl (\Res^{\wt M_{\phi_0}}_{J_1\cap \wt M[b]} (\phi_1))^{J_1}$.
Furthermore, as $B$ is $J_2$-stable, it follows that
\[
\bl (\Res^{\wt G_{\chi_0}}_{J_2}(\chi_1) )=\bl (\Res^{\wt G_{\chi_0}}_{J_1}(\chi_1))^{J_2}.
\]
As $H[B]=H\cap \wt G[B]=G(H\cap \wt M)[b]$ for each $G\leq H\leq \wt G$ it follows that
\[
J_2[B]\cap\wt M_{\phi_0}=G(J_2\cap \wt M)[b]\cap \wt M_{\phi_0}=(G\cap \wt M_{\phi_0})(J_2\cap \wt M)[b]=M(J_2\cap \wt M)[b]=(J_2\cap \wt M)[b]
\]
and so
$ \bl( \Res^{\wt M_{\phi_0}}_{{J_1\cap \wt M_{\phi_0}}} (\phi_1)
)$ is covered by $\bl (\Res^{\wt M_{\phi_0}}_{J_2\cap \wt M_{\phi_0}} (\phi_1))$.
Since 
$\bl( \Res^{\wt G_{\chi_0}}_{J_2}(\chi_1))$ is the unique block of $J_2$ covering 
$ \bl (\Res^{\wt M_{\phi_0}}_{J_1\cap \wt M_{\phi_0}} (\phi_1)
)^{J_1}$ and $ \bl (
\Res^{\wt M_{\phi_0}}_{{J_1\cap \wt M_{\phi_0}}} (\phi_1)
)$ is uniquely covered by $\bl (\Res^{\wt M_{\phi_0}}_{{J_2\cap \wt M_{\phi_0}}} (\phi_1))$, we see 
$\bl( \Res^{\wt G_{\chi_0}}_{J_2}(\chi_1))=
\bl (\Res^{\wt M_{\phi_0}}_{{J_2\cap \wt M_{\phi_0}}} (\phi_1))^{J_2}
$.

Now combining this equality with the proof of \cite[Proposition 4.2]{CabSpAMTypeA} verifies our statement.
\end{proof}

For the proof of the statements in \ref{thm11b} and \ref{thm13b} about the blockwise Alperin weight condition we point out the following property of $\Omega_\calB$, which is clear from its construction. 
\begin{cor}\label{cor_AWC}
	Assume that in the situation of Theorem \ref{IndAMCond}. If $\mathbb B\subset\Irr_0(\mathcal B)$ is a union of $\wt G$-orbits and $\mathbb B'\subset\Irr_0(\mathcal B')$ is a union of $\wt M$-orbits with $\wt\Omega(\Irr(\wt G\mid \mathbb B))=
	\Irr_0( \calB')\cap \Irr(\wt M\mid \mathbb B')$. Then $\Omega_\calB(\mathbb B)=\mathbb B'$.
	\end{cor}

\subsection{Application of inertial blocks} 
For Puig's notion of {\it inertial blocks} we refer to \cite[1.5]{Puig_nilpotent_extensions}.
Here are two propositions showing an instance where the McKay conjecture for a quotient group implies the Alperin-McKay conjecture for a certain block. It will later be used in the proof of Theorem \ref{thm13a}. The idea, inspired by \cite[\S 10]{SpExz}, is to use Clifford-theoretic properties of the Brauer correspondence for inertial blocks in order to verify that a block satisfies the Alperin-McKay Conjecture.
Recall that an {\it extension map} $\Lambda$ with respect to $X\lhd Y$ for $M\subseteq \Irr(X)$ is a map such that 
for every $\chi\in M$, $\Lambda(\chi)$ is an extension of $\chi$ to $Y_\chi$. 
For $L\lhd N$ and $\mathbb L\subset \Irr(L)$ recall $\Irr(N\mid \mathbb L):=\bigcup _{\chi \in \mathbb L}\Irr(N\mid \chi)$.
\begin{prop} \label{prop75}
	Let $L\lhd N$ and $C\in\Bl(N)$ with a defect group $D$ satisfying $\Cent_N(D)\leq L$. 
	Let $C_0\in\Bl(L)$, such that some $\wt C_0\in\Bl(N_{C_0}\mid C_0)\cap \Bl(N_{C_0}\mid D)$ satisfies $\wt C_0^N=C$. (Note that $C_0$ exists by \cite[9.14]{NavBl}.)
	Let $\mathbb{B}:=\left (\bigcup_{\kappa\in\Irr_0(B)} \Irr(\Res^N_L(\kappa)) \right ) \cap \Irr(C_0)$.
		Assume that every $\xi\in \mathbb B$.
		extends to $N_\xi$ and $N_\xi/L$ satisfies the McKay Conjecture for $\ell$. 
		Assume that $\wt C_0$ is inertial.
		Then the Alperin-McKay Conjecture holds for $C$.
\end{prop}
\begin{proof} 
	Let $C'$ be the Brauer correspondent in $\NNN_N(D)$ of $C$.
	
	According to \cite[9.14]{NavBl} induction defines a bijection between $\Irr_0(\wt C_0)$ and $\Irr_0(C)$. 
	Let $\Lambda$ be an extension map with respect to $L\lhd N$ for 
	$\left (\bigcup_{\kappa\in\Irr_0(B)} \Irr(\Res^N_L(\kappa)) \right ) \cap \Irr(C_0)$.
	The characters in $\Irr_0(\wt C_0)$ are of the form $\Ind_{N_{\xi}}^{N_{C_0}}(\Lambda(\xi)\eta)$ where $\xi\in\mathbb B$ with $D\leq N_\xi$ and $\eta\in\Irr(N_\xi/L)$ with $\ell\nmid \eta(1)$.
	Since $\Cent_N(D)\leq L$, all characters of this form belong to $\wt C_0$.
	Since $\wt C_0$ is the only block of $N_{C_0}$ that covers $C_0$ we have $\ell \nmid |N_{C_0}:LD|$ according to \cite[9.17]{NavBl}.

	Then $\{ \kappa \in \Irr(N_\xi\mid \xi)\mid \kappa(1)_\ell=\xi(1)_{\ell}\}$ corresponds to 
	 $\{ \eta \in \Irr(N_\xi/L )\mid \eta (1)_\ell=1\}=\Irr_{\ell'}(N_\xi/L)$. 
	 Let $M:=\NNN_{N_{C_0}}(LD)$.
	Now the McKay conjecture for $N_\xi/L$ 
	implies $|\Irr_{\ell'}(N_\xi/L)|= |\Irr_{\ell'} (M_\xi)|$. 
	The later set corresponds to $\{\kappa \in \Irr(M_\xi \mid \xi)\mid \kappa(1)_\ell=\xi(1)_{\ell}\}$. 
	Let $\wt C_0'$ be the block of $M$ with $(\wt C_0')^{N_{C_0}}=\wt C_0$. The above implies $|\Irr_0(\wt C_0')|=|\Irr_0(C)|$. 
	
	Recall that by assumption the block $C_1\in \Bl(LD\mid C_0)$ is inertial in the sense of \cite[1.5]{Puig_nilpotent_extensions}.
	Because of $\ell \nmid |N_{C_0}:LD|$ we see that $LD\lhd M$ with $\ell'$-index. 
	Then $\wt C_0'$ covers $C_1$ and is inertial according to \cite[Corrollary]{Zhou_pp_ext}. Accordingly $\wt C_0'$ is basic Morita equivalent in the sense of \cite[\S 7]{Puig_book} to the corresponding block $\wt c_0$ of its stabilizer subgroup $\NNN_{M}(D,{c_0})$, where $c_0\in \Bl(\Cent_M(D))$ that is covered by the Brauer correspondent $\wh c_0'\in \Bl(\NNN_M(D))$ of $\wt C_0'$. Note that the stabilizer subgroup from \cite[1.5]{Puig_nilpotent_extensions} is defined in terms of local pointed groups, while here we use the equivalent description of this group in terms of a maximal Brauer pairs, see \cite[40.13(d)]{Thevenaz_book}. A basic Morita equivalence implies $|\Irr_0(\wt C_0')|=|\Irr_0(\wt c_0)|$. By definition, $\wt c_0$ is the Fong-Reynolds correspondent of $C'$ with respect to $c_0$, i.e., $\wt c_0^{\NNN_{N}(D)}=C'$ and $|\Irr_0(\wt c_0)|=|\Irr_0(C')|$. This implies the statement. 
\end{proof}

For proving the blockwise Alperin weight conjecture we can prove the following analogue of the above. 
\begin{prop}
	Assume that $D$ is abelian in the situation of Proposition \ref{prop75}. 
	\begin{thmlist}
		\item \label{prop29a}
	Then $C$ satisfies the blockwise Alperin weight conjecture. 
	\item 	\label{prop29b}
	If $\mathbb B' \subset \left (\bigcup_{\kappa\in\Irr_0(B)} \Irr(\Res^N_L(\kappa)) \right ) \cap \Irr(C_0)$ forms an $N_{C_0}$-stable basic set of $C_0$ in the sense of \cite[14.3]{CabEnRedGp} whose corresponding $\ell$-modular decomposition matrix is unitriangular, then $|\Irr(N\mid \mathbb B')|=|\IBr(C)|$.
	\end{thmlist}
\end{prop}
\begin{proof}
	Since $D$ is abelian, $D\leq \Cent _N(D)\leq L$, hence $LD=L$ and $C_0$ is inertial. 
	Recall $\ell\nmid |N_{C_0}:L|$ by the 
	proof of Proposition \ref{prop75}, hence $M:=\NNN_{N_{C_0}}(LD)=N_{C_0}$. By the above proof the Fong-Reynolds correspondent $\wt C_0'\in \Bl(M)$ of $C$ covering $C_0$ is inertial, as well. We see that $|\IBr(C)|=|\IBr(\wt C_0')|$. 
	Let $\wt c_0'\in \Bl(\NNN_M(D))$ be the Brauer correspondent of $\wt C_0$. Now $\wt C_0'$ and $\wt c_0'$ are basic Morita equivalent and hence satisfy $|\IBr(\wt C_0')|=|\IBr(\wt c_0))|$.
	
	Since $\mathbb B'$ is an $N_{C_0}$-stable set and the $\ell$-modular decomposition matrix of $\Irr(C_0)$ is unitriangular with respect to $\mathbb B'_0$ there exists a $N_{C_0}$-equivariant bijection $\Upsilon:\mathbb B'\longrightarrow \IBr(C_0)$ such that $\Upsilon(\xi)$ has multiplicity one in $\xi^\circ$, where $\xi^\circ $ denotes the restriction of $\xi$ to $\ell$-regular elements. Let $\xi\in\mathbb B'$. Then
$N_{\xi}=N_{\Upsilon(\xi)}$ because of the equivariance of $\Upsilon$. Some extension of $\Upsilon(\xi)$ is a constituent of $\Lambda(\wt \xi)^\circ$, since $\Lambda(\wt \xi)^\circ$ has a constituent $\phi$ in $\IBr(N_\xi\mid \Upsilon(\xi))$ and
$\Res^{N_\xi}_L(\phi)$ is a summand of $ \Res^{N_\xi}_L(\Lambda(\wt \xi)^\circ)=\xi^\circ$. Let $\Lambda'$ be the $N_{C_0}$-equivariant extension map with respect to $L\lhd N_{C_0}$ for $\mathbb B'$ where for $\xi\in \mathbb B'$, $\Lambda'(\Upsilon(\xi))$ is defined as the unique extension of $\Upsilon(\xi)$ that is a constituent of $\Lambda(\xi)^\circ$. 
Recall $\ell\nmid |N_{C_0}:L|$. Then via $\Lambda'$ there is a correspondence between $\Irr(N\mid \mathbb B')$ and $\IBr(C)$.
\end{proof}

\section{$d$-Split Levi subgroups in type $\mathrm{A}_{n-1}$}
\label{dSplitLevi}

The conditions presented in Theorem~\ref{NewIndAmCond} are aimed specifically at groups of Lie type. We will apply it for simple groups $S$ such that $G=\SL_n(q)$ or $\SU_n(q)$ for $q$ a power of a prime $\not=\ell$. 

For that purpose it is important to study the Clifford theory arising from $d$-split Levi subgroups and their normalizers. The case of minimal $d$-split Levi subgroups is the subject of \cite[\S 5]{CabSpCharTypeA}. Here we treat the general case.
\subsection{Notation for type $\mathrm{A}_{n-1}$}
This section is used to fix some of the standard notation that will be used throughout.

Fix a prime $p$ and $q=p^m$ for some positive integer $m$.
Let $n\geq 2$. We denote
\[
{\bG}:=\SL_n(\ol {\F}_p)\leq \wt {\bG}:=\GL_n(\ol {\F}_p).
\]
Let $\wtbT$ the maximal torus of diagonal matrices in $\wt{\bG}$ and ${\bf T}:=\wtbT\cap {\bG}$.
The corresponding root system $\Gamma$ identifies with the subset $\{ e_i-e_j\mid 1\leq i,j\leq n, i\ne j\}$ of $\R^n$, where $(e_i)_{1\leq i\leq n}$ denotes the standard orthonormal basis \cite[1.8.8]{GLS3}.
For every root $\alpha\in\Gamma$, let ${\bf x}_{\alpha}(t)$ the matrix such that ${\bf x}_{e_i-e_j}(t)-{\rm Id}_n$ is the elementary matrix with entry $t$ at position $(i,j)$.
Set ${\bf n}_{\alpha}:={\bf x}_\alpha(1){\bf x}_{-\alpha}(-1){\bf x}_{\alpha}(1)$. 
For $1\leq i\leq n-1$, denote ${\bf n}_{e_{i+1}-e_i}\in \NNN_{\bG}({\bf T})$ by ${\bf n}_i$.

Set $\gamma_0: \wt{\bG}\rightarrow \wt{\bG}$ to be the automorphism defined by $g\mapsto (g^{tr})^{-1}$, where $g^{tr}$ denotes the transpose matrix of $g$. 
Let $v_0\in\wt{\bG}$ be the matrix with the entry $(-1)^{k+1}$ at position $(k,n+1-k)$ and $0$ elsewhere \cite[2.7]{GLS3}.
Then 
\[
\gamma:\wt{\bG}\rightarrow \wt{\bG}, \text{ } g\mapsto v_0\gamma_0(g)v_0^{-1},
\]
is a graph automorphism of $\wt{\bG}$ with $\gamma({\bf x}_{\alpha}(t))={\bf x}_{\gamma(\alpha)}((-1)^{|\alpha|+1}t)$ where $|e_i-e_j|:=|i-j|$ and $\gamma(e_i-e_j)=e_{n+1-j}-e_{n+1-i}$.
Also denote by \[
F_q:\wt{\bG}\rightarrow \wt{\bG}
\] the raising of matrix entries to the $q$-th power for $q=p^m$ ($m\geq 1$), so that $F_q=F_p^m$. For $\epsilon\in\{\pm 1\}$ set $F=\gamma^{\frac{1-\epsilon}{2}}\circ F_q$, 
$\wt{G}:=\wt{\bG}^F={\rm GL}_n({\epsilon}q)$ and $G:={\bG}^F=\SL_n(\epsilon q)$.
As $F_q$ commutes with $v_0$, it follows that $\bG^F$ is $\wt{\bG}$-conjugate to $\bG^{\gamma_0^{\frac{1-\epsilon}{2}}F_q}$, in particular they are isomorphic.

For each root $\alpha$, it is clear that ${\bf n}_{\alpha}$ is fixed by $F_p$.
Moreover, by considering the action of $\gamma_0$ on ${\bf n}_{\alpha}$, it can be seen that $\gamma_0$ also fixes each such element.

Let $V:=\langle {\bf n}_{\alpha}\mid \alpha\in \Gamma\rangle$, $H:={\bf T}\cap V$.\ 
 Let $e:=2|V|$ and define
\begin{equation}\label{eq:EAuto}
E:=C_{em}\times C_2
\end{equation}
which acts on ${\wt\bG}^{F_p^{em}}$ such that the generating element of the first factor acts by $F_p$ and the second factor acts by $\gamma$.
Denote by $\widehat{F}$ the element $\gamma\widehat{F}_p^m$ of $E$ inducing the automorphism $F$ on ${\bG}^{F_p^{em}}$.
Observe that $F^e=F_p^{em}$ and thus $G\leq {\bG}^{F_p^{em}}$, an $E$-stable subgroup.
For any $F$-stable subgroup ${\bf H}\leq {\bG}$ the normalizer in ${\bG}^F\rtimes E$ is well-defined and will be denoted by $({\bG}^F\rtimes E)_{\bf H}$.

\subsection{Construction of $d$-spit Levi subgroups}\label{ConstdSplitLevi}

The aim is to study how $\Aut ({\bG}^F)_\bL$ acts on $\Irr (\NNN_{\bG}(\bL)^F)$ for $d$-split Levi subgroups $\bL$ of ${\bG}$.
The groups are studied by considering the corresponding subgroups of $\bG^{vF}$ for some well chosen element $v\in {\bG}^{F_p}$.

Let ${\calV}$ be the $\R$-span of $\Gamma$ in $\R^n$ and fix $1\leq d\leq n$ and $d_0\geq 1$ the integer such that $\Phi_{d_0}(-x)=\pm\Phi_d(x)$, where $\Phi_d(x)$ denotes the $d^{\rm th}$-cyclotomic polynomial.
Set $a:=\left \lfloor{\frac{n}{d_0}}\right \rfloor$ and $$v:={\bf n}_1\dots {\bf n}_{d_0-1}{\bf n}_{d_0+1}\dots {\bf n}_{ad_0-1}$$ that is the product ${\bf n}_1\dots {\bf n}_{ad_0}$ where the $ {\bf n}_{id_0}$'s for $1\leq i\leq a$ are removed.
Then $w:=\rho(v)\in \NNN_{\bG}({\bT})/{\bT}\cong \Symm_n$ is a product of $a$ disjoint $d_0$-cycles and $v$ acts on $\Gamma$ by $v(e_i-e_j)=e_{w(i)}-e_{w(j)}$.
Let $\zeta$ be a primitive $d_0^{\rm th}$-root of unity in $\C$ and set ${\calV}(w,\zeta)$ to be the eigenspace of $w$ on ${\calV}$ with eigenvalue $\zeta$.
Let $\Gamma'$ be a $w$-stable parabolic root subsystem of $\Gamma$ such that 
\begin{equation}\label{eq:wStabRoot}
\left( {\calV}(w,\zeta)\cap \Gamma'^{\perp} \right)^{\perp}\cap \Gamma=\Gamma'. 
\end{equation}
Then the corresponding $d$-split Levi subgroup of $\wt{\bG}$ is given by
\[
\wtbL_{\Gamma'}:=\GenGp{\wtbT, {\bf X}_{\alpha}\mid \alpha\in\Gamma'}=\Cent_{\wt \bG}(\wt \bS) ,
\]
where $\wt{\bS}$ is a $vF$-stable $\Phi_d$-torus of $\wt{\bG}$.
In addition, taking ${\bS}$ and ${\bL}_{\Gamma'}$ to be the kernel of the determinant map restricted to $\wt{\bS}$ respectively $\wtbL_{\Gamma'}$, yields the corresponding $d$-split Levi subgroup of ${\bG}$.

\subsection{Identification by partitions}
A parabolic root subsystem $\Gamma'\subseteq \Gamma$ yields an equivalence relation on $\ul n:=\{1,\dots, n\}$ by saying $i\sim j$ if and only if $e_i-e_j\in \Gamma'\cup\{0\}$ and the equivalence classes provide a partition of $\ul n$.
Conversely, a partition $\lambda=(\lambda_1,\dots,\lambda_t)\vdash \ul n$ yields a parabolic root subsystem
\[
\Gamma_{\lambda}:=\{e_i-e_j\mid \{i,j\}\subseteq \lambda_k \text{ for some } k \}.
\]

\begin{lm}\label{wStableBij}
There is a bijection $\gamma$ between the set of $\sigma$-stable parabolic root subsystems of $\Gamma$ and the set of $w$-stable partitions of $\ul n$.
\end{lm}

\begin{prop}\label{PartwStableRootSubsys}
 The following are equivalent for a $w$-stable partition $\lambda$:
\begin{asslist}
\item $\Gamma_{\lambda}$ satisfies Equation {\rm (\ref{eq:wStabRoot})}.
\item $w$ acts on $\lambda$ with at most one fixed point and all other orbits have length $d_0$. 
\end{asslist}
\end{prop}
\begin{proof}

Assume $\Gamma'$ satisfies Equation {\rm (\ref{eq:wStabRoot})} and set $\lambda=\gamma(\Gamma')$.
Set $X={\calV}(w,\zeta)\cap (\Gamma')^\perp$ so that $\Gamma'=X^\perp\cap \Gamma$ and $e_i-e_j\in \Gamma'$ if and only if $v_i=v_j$ for all $v\in X$.
Let $\lambda_0:=\{i\mid e_i\in X^\perp \}$.
If $i\in\lambda_0$, and $e_i-e_j\in \Gamma'$, then $j\in \lambda_0$ and thus $\lambda_0\in\lambda$ and is fixed by $w$.
Take $\mu\in\lambda\setminus\{\lambda_0\}$ and $i\in \lambda_k$.
If $j=w^k(i)\in\mu$, then $v_i=v_j=\zeta^{-k}v_i$ for all $v\in X$ and so $i\in \lambda_0$, which is a contradiction.
This shows that all $w$-orbits on $\lambda\setminus\{\lambda_0\}$ have length $d_0$.

In the converse direction, assume that $w$ acts on $\lambda$ with at most one fixed point and all other orbits have length $d_0$.
Let $e_i-e_j\in\left( {\calV}(w,\zeta)\cap \Gamma_{\lambda}^{\perp} \right)^{\perp}\setminus\{\Gamma_{\lambda}\}$.
Then there must exist a root of the form $e_i-e_{w^k\cdot j}\in \Gamma_{\lambda}$, with $k\ne 0$.
Thus $v_j=v_i=v_{w^k\cdot j}$, which implies $\{i,j\}$ is contained within the fixed set of $\lambda$ giving a contradiction.
Thus $\Gamma_\lambda$ satisfies Equation (\ref{eq:wStabRoot}).
\end{proof}

\subsection{Some notation associated to partitions}\label{NotPart}

Assume $\lambda\vdash \underline{n}$ satisfies condition (ii) of Proposition~\ref{PartwStableRootSubsys}.
If $\lambda$ has a set fixed by $w$, then we denote this set by $\lambda_0$ and $\lambda':=\lambda\setminus\{\lambda_0\}$, otherwise $\lambda'$ coincides with $\lambda$.
Let $\ol{\mu}$ denote the $w$-orbit of $\mu\in\lambda'$ and $\ol{\lambda'}$ the set of all $w$-orbits.
As all $w$-orbits on $\lambda'$ have length $d_0$, for each $f\in \{1,\dots,n\}$ define $\lambda_f:=\{ \mu\in\lambda'\mid |\mu|=f\}$ and
\[
t_f:=\frac{|\lambda_f|}{d_0},
\]
so that $t:=\sum_{f=1}^n t_f$ equals the number of $w$-orbits on $\lambda'$.
In addition, for each $w$-orbit $\ol{\mu}\in\ol{\lambda'}$ define
\[
I_{\ol{\mu}}:=\{ i \mid \text{ there exists } \mu'\in\ol{\mu} \text{ such that } i\in\mu'\}.
\]

Take $\mu,\mu'\in\lambda'$ such that $|\mu|=|\mu'|=f$ but $\ol{\mu}\ne \ol{\mu}'$.
Set $I_{\ol{\mu}}=\{I_{\ol{\mu}}(1),\dots,I_{\ol{\mu}}(d_0f)\}$ with $I_{\ol{\mu}}(i)< I_{\ol{\mu}}(i')$ if and only if $i<i'$.
This labelling yields an order preserving bijection 
\begin{equation}\label{eq:BijMu}
\begin{array}{crcl}
\tau_{\ol{\mu},\ol{\mu}'}: & I_{\ol{\mu}} & \longrightarrow & I_{\ol{\mu}'}\\
 & I_{\ol{\mu}}(i) & \longmapsto & I_{\ol{\mu}'}(i).\\
\end{array}
\end{equation}
Furthermore, up to conjugation of $\lambda$ by an element of the centraliser in the symmetric group $\mathfrak{S}_n$ of $w$, it can be assumed that $\tau_{\ol{\mu},\ol{\mu}'}(w^k(\mu))\in \ol{\mu}'$.

\subsection{The structure of $d$-split Levi subgroups}

Let $\lambda\vdash \ul n$ be a partition satisfying condition (ii) of Proposition~\ref{PartwStableRootSubsys}. 
We will make use of the notation from Section~\ref{NotPart}.
For each $\mu\in \lambda$ define 
\[
\wtbT_{\mu}:=\{ \diag(z_1,\dots,z_n)\in \wtbT\mid z_j=1 \text{ if } j\not\in \mu\} \text{ and }
\wtbL_{\mu}:=\GenGp{ \wtbT_{\mu}, X_{\alpha}\mid \alpha\in \Gamma_{\mu}}.
\]
Then
\[
\wtbL_\lambda:=\GenGp{ \wtbT, {\bf X}_{\alpha}\mid \alpha\in\Gamma_{\lambda}}=\wtbL_{\lambda_0}\times \prod_{\mu\in\lambda'} \wtbL_{\mu}\cong \GL_{|\lambda_0|}(\ol{\F}_p)\times \prod_{f=1}^n \GL_{f}(\ol{\F}_p)^{d_0t_f}.
\]
Furthermore, by construction $F_p(\wtbL_{\mu})=\wtbL_{\mu}$ and $v(\wtbL_{\mu})=\wtbL_{w(\mu)}$.
Therefore
\[
\wt{L}_\lambda:=\wtbL_\lambda^{vF_q}\cong \GL_{|\lambda_0|}(\epsilon q)\times \prod_{f=1}^n \GL_{f}(\epsilon q^{d_0})^{t_f}.
\]
In addition $\bL_{\lambda}$ and $L_{\lambda}$ are the kernel of the determinant maps on $\wtbL_{\lambda}$ and $\wt{L}_{\lambda}$ respectively.

\subsection{The structure of $N_{\lambda}/L_{\lambda}$}

Given ${\bG}$ and an $F$-stable Levi subgroup $\bL$, let $W_{\bG}$ and $W_\bL$ denote the Weyl groups of ${\bG}$ and $\bL$ respectively.
In addition, set ${\bf N}:=\NNN_{\bG}(\bL)$ so that $N={\bf N}^{vF}$. 
The relative Weyl group of $\bL$ in ${\bG}$ (see \cite[Section 4]{BMGenHC}) is defined to be
\[
W_{\bG}(\bL):={\bf N}/\bL\cong \NNN_{W_{\bG}}(W_\bL)/W_\bL.
\]
If $\bL$ is $vF$-stable, it follows that ${\bf N}$ is also $vF$-stable, and therefore $vF$ induces an automorphism $\sigma$ on $W_{\bG}(\bL)$.
Furthermore, by the Lang-Steinberg Theorem
\[
N/L\cong \Cent_{W_{\bG}(\bL)}(\sigma).
\]

For a partition $\lambda\vdash \ul n$ set $W_{\lambda}=\cap_{\mu\in\lambda} W_\mu$ the intersection of the setwise stabilisers of $\mu\in\lambda$ in $W\cong \Symm_n$.
If $\lambda$ is a partition satisfying condition (ii) of Proposition~\ref{PartwStableRootSubsys} and $\bL:=\bL_{\lambda}$, then
\[
N/L\cong \{x\in \Symm_n\mid \lambda^x=\lambda \text{ and } w^xw^{-1}\in W_{\lambda}\}/W_{\lambda}. 
\]
If $x\in \Symm_n$ fixes $\lambda$ and $w^xw^{-1}\in W_{\lambda}$, then $w^xw^{-1}$ fixes $\lambda_0$ and so $(\lambda_0)^x$ is fixed by $w$.
Therefore $(\lambda_0)^x=\lambda_0$. 
Furthermore, as $|\mu^x|=|\mu|$ for all $\mu\in \lambda'$, it follows that 
\[
\Cent_{W_{{\bG}}(\bL_{\lambda})}(w)=\prod\limits_{e=1}^n \Cent_{{\rm Sym}(\lambda_f)}(w),
\]
where we recall that $\lambda_f:=\{\mu\in\lambda\mid |\mu|=f\}$. 

\begin{prop}\label{RelWeylGpPart}
Let $\lambda$ be a $w$-stable partition of $\ul n$ satisfying condition (ii) of Proposition~\ref{PartwStableRootSubsys}.
Then for $L=\bL_\lambda^{vF}$ and $N=\NNN_{\bG}(\bL_\lambda)^{vF}$, 
\[
N/L \cong \prod\limits_{e=1}^n C_{d_0}\wr \Symm_{t_f}.
\]
\end{prop}

\subsection{The structure of $N_{\lambda}$}

Let $\lambda\vdash \ul n$ satisfying Proposition~\ref{PartwStableRootSubsys} and recall that $\lambda_f=\{ \mu\in\lambda' \mid |\mu|=f\}$ and set $\ol{\lambda}_f$ to be the set of $w$-orbits on $\lambda_f$, so $t_f=|\ol{\lambda}_f|$. 
Then for ${\rm Sym}(\ol{\lambda}')_{\{ \ol{\lambda}_1,\dots,\ol{\lambda}_n\} }$, the intersection of the setwise stabilisers ${\rm Sym}(\ol{\lambda}')_{\ol{\lambda}_f}$, Proposition~\ref{RelWeylGpPart} shows that
\[
N/L\cong (C_{d_0})^t \rtimes {\rm Sym}(\ol{\lambda}')_{\{ \ol{\lambda}_1,\dots,\ol{\lambda}_n\} }.
\]
We construct subgroups $V_0,S\leq V\cap N$ such that $\rho(V_0)=C_{d_0}^t$ and $\rho(S)={\rm Sym}(\ol{\lambda}')_{\{ \ol{\lambda}_1,\dots,\ol{\lambda}_n\} }$.

\subsubsection{The construction of $V_0$}

Let $\ol{\mu}\in\ol{\lambda}'$ and recall the elements ${\bf n}_i$ as defined in Section \ref{ConstdSplitLevi}.
Then define
\[
v_{\ol{\mu}}:=\prod\limits_{i\in I_{\ol{\mu}}\setminus d\Z}{\bf n}_i,
\]
where the product is taken with respect to the natural ordering "$\leq$ " on $\N$.
Then by construction $[v_{\ol{\mu}},v]=1$ and the two elements $v_{\ol{\mu}}$ and $v$ induce the same action on $\wt{ L}_{\ol{\mu}}$ while $[v_{\ol{\mu}},\wt{L}_{\ol{\mu}'}]=1$ whenever $\ol{\mu}'\ne \ol{\mu}$.
Set
\[
V_f:=\GenGp{ v_{\ol{\mu}}\mid |\mu|=f}
\text{ and }
V_0:=\GenGp{ v_{\ol{\mu}}\mid \ol{\mu}\in\ol{\lambda'}}.
\]
Then $\rho(V_f)=C_{d_0}^{t_f}$ and $\rho(V_0)=C_{d_0}^t$.

\subsubsection{The construction of $S$}

Let $\ol{\mu}\ne \ol{\mu}'\in\ol{\lambda}_f$ and $\tau_{\ol{\mu},\ol{\mu}'}$ the bijection between $I_{\ol{\mu}}$ and $I_{\ol{\mu}'}$ from Equation (\ref{eq:BijMu}).
Define an element
\[
s_{\ol{\mu},\ol{\mu}'}:=\prod\limits_{i\in I_{\ol{\mu}}} {\bf n}_{e_i-e_{\tau_{\ol{\mu},\ol{\mu}'}(i)}}(1).
\]
Then $s_{\ol{\mu},\ol{\mu}'}$ acts on $\wt{L}$ by permuting $\wt{L}_{\ol{\mu}}$ and $\wt{L}_{\ol{\mu}'}$, while $[s_{\ol{\mu},\ol{\mu}'},\wt{L}_{\ol{\mu}''}]=1$ for every $\mu''\not\in \ol{\mu}$ or $\ol{\mu}'$.
Let
\[
S_f:=\GenGp{ s_{\ol{\mu},\ol{\mu}'}\mid \ol{\mu},\ol{\mu}'\in\ol{\lambda}_f}
\text{ and } 
S=\GenGp {S_f\mid 1\leq f\leq n}.
\]
Then by construction $\rho(S_f)={\rm Sym}(\ol{\lambda}_f)$, and $\rho(S)={\rm Sym}(\ol{\lambda}')_{\{ \ol{\lambda}_1,\dots,\ol{\lambda}_n\} }$.
Furthermore, $v_{\ol{\mu}}^{s_{\ol{\mu},\ol{\mu}'}}=v_{\ol{\mu}'}$ so $S$ is in the normalizer of $V_0$ and $[s_{\ol{\mu},\ol{\mu}'},v]=1$.

\begin{prop}\label{StructN}
Let $\lambda$ be a $w$-stable partition satisfying condition (ii) of Proposition~\ref{PartwStableRootSubsys}.
Let $L=\bL_\lambda^{vF}$ and $N=\NNN_{\bG^{vF}}(\bL_\lambda)$.
Then there exist subgroups $V_0,S\leq V\cap N$ such that
\[
N=LV_0S
\]
where $\rho(V_0)=C_{d_0}^t$ and $\rho(S)=\rho(S)={\rm Sym}(\ol{\lambda}')_{\{ \ol{\lambda}_1,\dots,\ol{\lambda}_n\} }$.
\end{prop}

\subsubsection{Constructing an extension map with respect to $L\cap S\lhd S$}
Let $f\in\{1,\dots, n\}$ and write $\ol{\lambda}_f=\{\ol{\mu}_1,\dots,\ol{\mu}_{t_f}\}$.
Set $V^{(f)}:=\GenGp{ {\bf n}_{e_i-e_j}(1)\mid 1\leq i,j\leq t_f} \leq \SL_{t_f}(p)$. 
There is a natural isomorphism from $S_f$ to $V^{(f)}$ given by $s_{\ol{\mu}_i,\ol{\mu}_j}\mapsto {\bf n}_{e_i-e_j}(1)$.
Moreover, for $T_f$ the diagonal maximal torus of $\SL_{t_f}(p)$, it follows that the image of $L\cap S_f$ is $T_f\cap V^{(f)}$.
\begin{cor}\label{ExtK0ToS}
There exists an extension map with respect to $L\cap S\lhd S$.
\begin{proof}
This follows from observing that $$L\cap S=\GenGp{s_{\ol{\mu},\ol{\mu}'}^2\mid \ol{\mu},\ol{\mu}'\in \ol{\lambda}_f \text{ for some }f}=\prod\limits_{f=1}^n L\cap S_f$$ and that there is an extension map with respect to $T_f\cap V^{(f)}\lhd V^{(f)}$ by \cite[Proposition 5.5]{CabSpCharTypeA}.
\end{proof}
\end{cor}

\subsection{Characters of $d$-split Levi subgroups}

Define $L_0:=[L,L]=[\wt{L},\wt{L}]$, $L_f:=[\wt{L}_f,\wt{L}_f]$ and $L_{\ol{\mu}}:=[\wt{L}_{\ol{\mu}},\wt{L}_{\ol{\mu}}]$ so that
\[
L_0=L_{\lambda_0}\times \prod_{f=1}^{n} L_f=L_{\lambda_0}\times \prod_{f=1}^{n} \left( \prod_{\ol{\mu}\in\ol{\lambda}_f} L_{\ol{\mu}}\right) \cong \SL_{|\lambda_0|}(\epsilon q)\times \prod_{f=1}^n \SL_{f}(\epsilon q^{d_0})^{t_f}.
\]

\begin{lm}\label{CliffCorrRest}
Let $\chi\in \Irr (L)$ and $\chi_0\in \Irr (\Res^L_{L_0}(\chi))$, then the Clifford correspondent of $\chi$ over $\chi_0$ restricts irreducibly to $\chi_0$.
\begin{proof}

Write 
\[
\chi_0=\chi_{\lambda_0}\times \prod_{\ol{\mu}\in\ol{\lambda}'} \chi_{\ol{\mu}},
\]
with $\chi_{\ol{\mu}}\in \Irr (L_{\ol{\mu}})$ and $\chi_{\lambda_0}\in \Irr (L_{\lambda_0})$.
Then
\[
\wt{L}_{\chi_0}=(\wt{L}_{\lambda_0})_{\chi_{\lambda_0}}\times \prod_{\ol{\mu}\in\ol{\lambda}'} (\wt{L}_{\ol{\mu}})_{\chi_{\ol{\mu}}}.
\]
Each factor $\wt{L}_{\ol{\mu}}/L_{\ol{\mu}}$ and $\wt{L}_{\lambda_0}/L_{\lambda_0}$ is cyclic, therefore $\chi_0$ extends from $L_0$ to $\wt{L}_{\chi_0}$ and hence to $L_{\chi_0}$.
Moreover as $L_{\chi_0}/L_0$ is abelian, it follows from Gallagher's Theorem that $\Res^{\wt L_\chi}_{L_0}(\wt{\chi})=\chi_0$. 
\end{proof}
\end{lm}

\begin{lm}\label{InertiaL}
Let $\chi\in \Irr (L)$ and $\chi_0\in \Irr (\Res^L_{L_0}(\chi))$.
Then $\wt{L}_{\chi}=L\wt{L}_{\chi_0}$.
\begin{proof}
As $\wt{L}_\chi\leq L\wt{L}_{\chi_0}$, it suffices to show that $\chi$ extends to $L\wt{L}_{\chi_0}$.
Let $\psi\in \Irr (L\wt{L}_{\chi_0}\mid \chi)$.
Then $\chi_0\in \Irr (\Res^{L\wt{L}_{\chi_0}}_{L_0}(\psi))$ and hence there exists a unique $\wt{\psi}\in \Irr (\wt{L}_{\chi_0}\mid \chi_0)$ such that $\Ind_{\wt{L}_{\chi_0}}^{L\wt{L}_{\chi_0}}(\wt{\psi})=\psi$.
However, $\wt{\psi}$ restricts to an irreducible character of $L_{\chi_0}$ as $\wt{\psi}$ restricts to $\chi_0$ by Lemma~\ref{CliffCorrRest}.
Therefore $\Ind_{L_{\chi_0}}^L(\Res_{L_{\chi_0}}^{\wt{L}_{\chi_0}}(\wt{\psi}))$ is irreducible and thus $\Res_{L}^{L\wt{L}_{\chi_0}}(\psi)=\chi$. 
\end{proof}
\end{lm}

Let us recall the fundamental property of stabilizers of characters in type $\mathrm{A}$.
\begin{thm}[Cabanes-Sp\"ath,{\cite[Thm 4.1]{CabSpCharTypeA}}]\label{StarCondSL} Let $E$ be defined as in Equation (\ref{eq:EAuto}).
If $\psi\in \Irr (\wt{\bf G}^F)$, then there exists a $\psi_0\in \Irr ({\bf G}^F\mid \psi)$ such that 
\[
(\wt{\bf G}^F\rtimes E)_{\psi_0}=(\wt{\bf G}^F)_{\psi_0}\rtimes E_{\psi_0}
\]
and $\psi_0$ extends to a character of ${\bf G}^F\rtimes E_{\psi_0}$.
\end{thm}

We apply this result to the Levi subgroups $L_{\ol{\mu}}$.
That is, set $\mathbb{T}_{\ol{\mu}}\subseteq \Irr (L_{\ol{\mu}})$ to be an $(E\langle v_{\ol{\mu}}\rangle )$-transversal on $\Irr (L_{\ol{\mu}})$ satisfying Theorem~\ref{StarCondSL} with respect to $(L_{\ol{\mu}}$ and $(\wt{L}_{\ol{\mu}}$.
Furthermore, each $\mathbb{T}_{\ol{\mu}}$ is chosen so that if $\theta\in \mathbb{T}_{\ol{\mu}}$ and $s_{\ol{\mu},\ol{\mu}'}\in S_f$ then $\theta^{s_{\ol{\mu},\ol{\mu}'}}\in \mathbb{T}_{\ol{\mu}'}$.
Then we can define $\mathbb{T}\subseteq \Irr (L_0)$ by
\begin{equation}\label{eq:CharTrans}
\mathbb{T}:=\{ \chi_{\lambda_0}\times \prod_{\ol{\mu}\in\ol{\lambda}'} \chi_{\ol{\mu}} \mid \chi_{\lambda_0}\in \mathbb{T}_{\lambda_0} \text{ and } \chi_{\ol{\mu}}\in \mathbb{T}_{\ol{\mu}} \}.
\end{equation}

\begin{lm}\label{InterialNChi0}
Let $\chi_0\in \mathbb{T}\subset \Irr (L_0)$ from Equation (\ref{eq:CharTrans}) and $V_0$, $S$ as in Proposition~\ref{StructN}.
Then
\[
(\wt{L}V_0ES)_{\chi_0}=\wt{L}_{\chi_0}(V_0E)_{\chi_0}S_{\chi_0}.
\]
In particular, for $\wh N:=NE$ it follows that
\[
\wt{N}_{\chi_0}=\wt{L}_{\chi_0}V_{0,\chi_0}S_{\chi_0} \text{ , } N_{\chi_0}=L_{\chi_0} V_{0,\chi_0}S_{\chi_0} \text{ and } \wh{N}_{\chi_0}=L_{\chi_0} (V_{0}E)_{\chi_0}S_{\chi_0}.
\]

\begin{proof}

Take $x\in \wt{L}V_0E$ and $s\in S$.
If $xs\in N_{\chi_0}$ and $\ol{\mu}\in\ol{\lambda}'$, then $\chi_{\ol{\mu}}^x$ and $\chi_{\rho(s)(\ol{\mu})}$ give the same character on $L_{\rho(s)(\ol{\mu})}$.
Thus by the choice of characters in $\mathbb{T}$, it follows that $\chi_{\ol{\mu}}^{s}=\chi_{\rho(s)(\ol{\mu})}$.
Therefore $s\in S_{\chi_0}$ and 
\[
(\wt{L}V_0ES)_{\chi_0}=(\wt{L}V_0E)_{\chi_0}S_{\chi_0}.
\]
Furthermore, the choice of $\mathbb{T}$ implies that $(\wt{L}V_0E)_{\chi_0}=\wt{L}_{\chi_0}(V_0E)_{\chi_0}$.
\end{proof}
\end{lm}

\begin{cor}\label{IntCond}
Let $\chi\in \Irr (L)$.
Then $\wt{N}_{\chi}=\wt{L}_{\chi}N_{\chi}$.
Moreover, there exists $\mathcal{T}$ an $\wt{L}$-transversal on $\Irr (L)$ such that 
\[
(\wt{N}\wh{N})_{\chi}=\wt{N}_{\chi}\wh{N}_{\chi}
\]
for each $\chi\in \mathcal{T}$. 
\begin{proof}
It suffices to prove these properties hold for an $\wt{L}\wh{N}$-transversal of $\Irr (L)$.
In particular, it can be assumed that there exists a $\chi_0\in \Irr (\Res_{L_0}^L(\chi))\cap \mathbb{T}$.
Let $\wt{\chi}_0\in \Irr (L_{\chi_0}\mid \chi_0)$ be the Clifford correspondent inducing to $\chi$.
Hence by Lemma~\ref{InterialNChi0} and Lemma~\ref{InertiaL}
\[
\wt{N}_{\chi}=L(\wt{N}_{\chi_0})_{\wt{\chi}_0}=L(\wt{L}_{\chi_0}V_{0,\chi_0}S_{\chi_0})_{\wt{\chi}_0}=\wt{L}_{\chi}(V_{0,\chi_0}S_{\chi_0})_{\wt{\chi}_0}=\wt{L}_{\chi}N_{\chi}.
\]
This proves the first statement.
Similarly by Lemma~\ref{InterialNChi0}
\[
(\wt{N}_{\chi_0})_{\wt{\chi}_0}(\wh{N}_{\chi_0})_{\wt{\chi}_0}= \left( \wt{L}_{\chi_0}(V_{0,\chi_0}S_{\chi_0})_{\wt{\chi}_0} \right)\left( L_{\chi_0}(\wh{V}_{0,\chi_0}S_{\chi_0})_{\wt{\chi}_0} \right)=\wt{L}_{\chi_0}\left( \wh{V}_{0,\chi_0}S_{\chi_0} \right)_{\wt{\chi}_0}= \left( (\wt{N}\wh{N})_{\chi_0}\right)_{\wt{\chi}_0}.
\]
Thus the second statement follows.
\end{proof}
\end{cor}

\section{The Clifford theory for $d$-split Levi subgroups}
\label{CliffThdSpLevi}

In \cite{CabSpIMTypeC} the authors provide a criterion to check property $2.2(iv)$ of \cite{SpIMDefChar} which arises in the inductive McKay condition.
In this section we consider the generalisation of this criterion to the case of a $d$-split Levi subgroup, which will be used to prove Theorem~\ref{StarConddSpLevi}.
\subsection{Clifford-theoretic tools}
There are two technical statements that will be applied in our context but have appeared in the context of the McKay Conjecture. 
\begin{thm}\label{MainThmReq}
Let $d$ be a positive integer and $\wtbL$ a $vF$-stable $d$-split Levi subgroup of $\wt{\bG}$.
Assume the groups $N:=\NNN_{{\bG}}(\bL)^{vF}$, $\wt{N}:=\NNN_{\wt{\bG}}(\bL)^{vF}$ and $\wh{N}:=(\Cent_{{\bG}^{F_0^{em}}E}(v\wh{F}))_\bL$ satisfy the following conditions:
\begin{asslist}
\item There exists some set $\mathcal{T}\subseteq \Irr (L)$, such that 
\begin{enumerate}
\item $\wt{N}_{\xi}=\wt{L}_{\xi}N_{\xi}$ for every $\xi\in\mathcal{T}$,
\item $(\wt{N}\wh{N})_{\Ind_L^N(\xi)}=\wt{N}_{\Ind_L^N(\xi)}\wh{N}_{\Ind_L^N(\xi)}$ for every $\xi\in \mathcal{T}$, and 
\item $\mathcal{T}$ contains some $\wh N$-stable $\wt{L}$-transversal of $\Irr (L)$.
\end{enumerate}

\item There exists an extension map $\Lambda$ with respect to $L\lhd N$ such that 
\begin{enumerate}
\item $\Lambda$ is $\wh{N}$-equivariant.
\item Every character $\xi\in \mathcal{T}$ has an extension $\wh{\xi}\in\Irr(\wh{N}_{\xi})$ with $\Res_{\NNN_{\xi}}^{\wh{N}_{\xi}}(\wh{\xi})=\Lambda(\xi)$ and $v\wh{F}\in \ker(\wh{\xi})$.
\end{enumerate}

\item Let $W_d:=N/L$ and $\wh{W}_d:=\wh{N}/L$.
For $\xi \in \Irr (L)$ and $\wt{\xi}\in \Irr (\wt{L}_\xi\mid \xi)$ let $W_{\wt{\xi}}:=N_{\wt{\xi}}/L$, $W_\xi:=N_{\xi}/L$, $K:=\NNN_{W_d}(W_\xi,W_{\wt{\xi}})$ and $\wh{K}:=\NNN_{\wh{W}_d}(W_\xi,W_{\wt{\xi}})$.
Then there exists for every $\eta_0\in\Irr (W_{\wt{\xi}})$ some $\eta\in \Irr (W_{\xi}\mid \eta_0)$ such that 
\begin{enumerate}
\item $\eta$ is $\wh{K}_{\eta_0}$-invariant.
\item If $D$ is non-cyclic, $\eta$ extends to some $\wh{\eta}\in \Irr (\wh{K}_{\eta})$ with $v\wh{F}\in \ker (\wh{\eta})$.
\end{enumerate}
\end{asslist}
Then:
\begin{thmlist}
\item For every $\chi\in \Irr (\wt{N})$ there exists some $\chi_0\in \Irr (N\mid \chi)$ such that
\begin{enumerate}
\item $(\wt{N}\wh{N})_{\chi_0}=\wt{N}_{\chi_0}\wh{N}_{\chi_0}$, and
\item $\chi_0$ has an extension $\wt{\chi}_0$ to $\wh{N}_{\chi_0}$ with $v\wh{F}\in \ker (\wt{\chi})$. 
\end{enumerate}
\item Moreover, there exists some $\wh N$-equivariant extension map with respect to $\wt L \lhd \wt N$ that is compatible with $\Irr(\wh N/N)$.
\end{thmlist}
\end{thm}

\begin{proof}
The proof of part (a) is the same as the proof of \cite[Theorem 4.3]{CabSpIMTypeC}. 
Part (b) uses the assumption that $\mathcal T$ is $\wh N$-stable. The proof follows from the considerations as in the proof of Theorem~4.2 of \cite{CS_typeBE}. 
\end{proof}

The following result helps to construct extensions in our context. 
Recall that for a finite group $X$ and $\xi\in\Irr (X)$ there is an associated subgroup $\Zent(\xi):=\{ z\in X\mid |\xi(z)|=\xi(1) \}$, see \cite[2.26]{IsaChTh}. 
\begin{prop}[{\cite[Prop. 2.6]{SpaethtypeD}}]
	\label{prop_tool}
	Let $K\lhd M$, $K_0\lhd M$ with $K_0\leq K$ and $M$ a finite group.
	Let $\xi\in\Irr(K)$ a character with $\xi_0:= \Res^K_{K_0}(\xi)\in\Irr(K_0)$. Assume 
	\begin{asslist}
		\item \label{prop32i} $K=\Zent(\xi)K_0$;
		\item assume there exists some group $V\leq M$ such that
		\begin{enumerate}
			\item 	\label{prop23ii1}
			$M=KV$ and $H:=V\cap K\leq \Cent_K (K_0)$; and 
			\item \label{prop23ii2}
			there is a $\zeta\in\Irr(\Res_H^K(\xi))$ which extends to some $\wt \zeta\in\Irr(V_\zeta)$;
		\end{enumerate}
		\item \label{prop_toolext} $\xi_0$ extends to $K_0\rtimes \epsilon(V_{\xi_0})$, where 
		$\epsilon: V \rightarrow V/H$ is the canonical epimorphism. 
	\end{asslist}
	Then there exists an extension of $\xi$ to $M_\xi$ that is afforded by a representation $\wt \cD$ satisfying 
	\begin{align}\label{wtcD} 
	\wt	\cD( kv)&= \wt \zeta(v)\cD' (\rho(v)) \mathcal D(k) \text{ for every }
	k\in K \text{ and }v\in V_\la ,\end{align}
	where $\mathcal D$ is a representation of $K$ affording $\xi$, and $\cD'$ is a representation of $K_0\rtimes \epsilon(V)_{\xi_0}$ extending $ \cD_{K_0}$. 
\end{prop}

\subsection{An extension map with respect to $L\lhd N$}
Let $\lambda\vdash \ul n$ satisfying condition (ii) of Proposition~\ref{PartwStableRootSubsys}.
Set $\wt{L}:=\wtbL_{\lambda}^{vF}$, $L:=\wt{L}\cap \bG^{vF}$ and $N=\NNN_{\bG^{vF}}(\bL)$.
All the notations from Section~\ref{dSplitLevi} will be used without further reference.
The aim in this section is to prove Theorem~\ref{AMTypeA} via the following stronger statement.
\begin{thm}\label{NEquivExt}
Let $\wh{N}:=NE$.
Then there exists an extension map $\Lambda$ with respect to $L\lhd N$ such that
\begin{asslist}
\item $\Lambda$ is $\wh{N}$-equivariant,
\item there exists a set $\mathcal{T}\subseteq \Irr (L)$ which contains an $\wt{L}$-transversal of $\Irr (L)$ such that if $D$ is non-cyclic, then every character $\chi\in \mathcal{T}$ has an extension $\wh{\chi}\in\Irr (\wh{N}_{\chi})$ with $\Res_{N_{\chi}}^{\wh{N}_{\chi}}(\wh{\chi})=\Lambda(\chi)$ and $v\wh{F}\in \ker (\wh{\chi})$.
\end{asslist}
\end{thm}

Clearly $N$ is normalised by $\wt{L}$ and therefore it suffices to produce an extension for characters in a $\wt{L}$-transversal of $\Irr (L)$.
Furthermore, to prove Theorem~\ref{NEquivExt} it suffices to define $\Lambda$ on a $\wh{N}$-transversal in $\Irr (L)$ and the remaining values of $\Lambda$ can be constructed using $\wh{N}$-conjugation.
In particular for the remainder of this section we assume that $\chi\in\Irr (L)$ and $\chi_0\in \Irr (\Res^L_{L_0}(\chi))\cap \mathbb{T}$, where $\mathbb{T}$ is taken from Equation (\ref{eq:CharTrans}).
Let $\wt{\chi}_0\in \Irr (L_{\chi_0}\mid \chi_0)$ with $\Ind_{L_{\chi_0}}^L(\wt{\chi}_0)=\chi$.
Then Lemma~\ref{CliffCorrRest} shows that $\wt{\chi}_0$ restricts to $\chi_0$. 

Consider the structure of $\wh{N}_{\chi}$.
As $L_{\chi_0}\lhd L$ and $\wh{N}_{\chi}$ permutes the $L$-conjugates of $\chi_0$, it follows that $L_{\chi_0}\lhd \wh{N}_{\chi}$.
Hence $\wh{N}_{\chi}\leq L(\wh{N}_{\chi_0})_{\wt{\chi}_0}$. 
However $\wt{\chi}_0$ induces to $\chi$ and $L_{\chi_0}\lhd \wh{N}_{\chi_0}$, therefore $(\wh{N}_{\chi_0})_{\wt{\chi}_0}\leq \wh{N}_{\chi}$.
Thus
\[
\wh{N}_{\chi}=L(\wh{N}_{\chi_0})_{\wt{\chi}_0}.
\]
Furthermore for $\wh{V}_0:=V_0E$, Lemma~\ref{InterialNChi0} shows that
\[
(\wh{N}_{\chi_0})_{\wt{\chi}_0}=L_{\chi_0}\left( (\wh{V}_0)_{\chi_0}S_{\chi_0}\right)_{\wt{\chi}_0}.
\]

\subsubsection{A useful subgroup of $L_{\chi_0}$ containing $L_0$}

As $L_0\cap S$ can be a proper subgroup of $L\cap S$, we need to consider a subgroup $L_0\leq K\leq L_{\chi_0}$ such that $K\cap S=L\cap S$.
Define $-{\rm Id}_{\wt{L}_{\ol{\mu}}}$ to be the unique central element of order two in $\wt{L}_{\ol{\mu}}$ and set
\[
Z_2:=\GenGp{ -{\rm Id}_{\wt{L}_{\ol{\mu}}} \mid \ol{\mu}\in \ol{\lambda}}.
\]
In addition define $K_{\ol{\mu}}:=\GenGp{ L_{\ol{\mu}}, L_{\ol{\mu}}\cap Z_2}$ and
\[
\wt{K}:=\GenGp{ L_0,Z_2}=K_{\lambda_0}\times \prod_{\ol{\mu}\in\ol{\lambda}'} K_{\ol{\mu}}.
\]
Then for $K:=L\cap \wt{K}=\ker(\det_{\wt{K}})$, it follows that $K\cap S=L\cap S$.
Set $\psi:=\Res_{K}^{L_{\chi_0}}(\wt{\chi}_0)\in \Irr (K)$ and let $\wt{\psi}\in \Irr (\wt{K})$ be an extension of $\psi$.
The restriction of $\wt{\psi}$ to $Z_2$ determines the extension of $\chi_0$ to $\wt{K}$.
However $[Z_2,\wh{V}_0]=1$ and therefore $\wh{V}_0$ fixes the restriction of $\wt{\psi}$ to $Z_2$.
Thus $(\wh{V}_0)_{\wt{\psi}}=(\wh{V}_0)_{\chi_0}$ and hence
\[
(\wh{V}_0)_{\psi}=(\wh{V}_0)_{\chi_0}.
\]

\begin{lm}
Let $M:=K\wh{V}_0S$.
If $\psi$ extends to $M_{\psi}$, then $\chi$ extends to $\wh{N}_{\chi}$.
\begin{proof}
By construction $\wh{N}_{\chi}=L(L_{\chi_0}M_{\chi_0})_{\wt{\chi}_0}$.
Moreover $L_{\chi_0}\cap (M_{\chi_0})_{\wt{\chi}_0}=K$ and $(M_{\chi_0})_{\wt{\chi}_0}\leq M_{\psi}$.
Thus an extension of $\psi$ to $M_{\psi}$ gives an extension $\phi$ of $\wt{\chi}_0$ to $(L_{\chi_0}M_{\chi_0})_{\wt{\chi}_0}$ by \cite[Lemma 4.1]{SpSylowTori2}.
Moreover, by the Mackey formula, 

\[
\Res_L^{\wh{N}_{\chi}}\left( \Ind_{(L_{\chi_0}M_{\chi_0})_{\wt{\chi}_0}}^{\wh{N}_{\chi}}(\phi)\right) =\Ind_{L_{\chi_0}}^L(\wt{\chi})=\chi.
\qedhere
\]
\end{proof}
\end{lm}

\subsubsection{Extending $\psi$ to $M_{\psi}$}

First we observe that $M_{\psi}\leq M_{\chi_0}$ and so
\[
M_{\psi}=K(\wh{V}_0)_{\psi}S_{\psi}.
\]

For each $\ol{\mu}\in\ol{\lambda}$, define a group 
\[
E_{\ol{\mu}}:=\GenGp{ \wh{F}_{\ol{\mu}},\gamma_{0,\ol{\mu}}}
\]
where $\wh{F}_{\ol{\mu}}$ and $\gamma_{0,\ol{\mu}}$ act as $\wh{F}$ and $\gamma_0$ on $\wt{L}_{\ol{\mu}}$, while $[E_{\ol{\mu}},\wt{L}_{\ol{\mu}'}]=1$ whenever $\ol{\mu}\ne \ol{\mu}'$.
Then the group
\[
E_0:=E_{\lambda_0}\times\prod\limits_{\ol{\mu}\in\ol{\lambda}'} E_{\ol{\mu}}
\]
contains $E$ as a diagonally embedded subgroup.
Set $\wh{V}_{\ol{\mu}}=V_{\ol{\mu}}E_{\ol{\mu}}$.
Then by Theorem~\ref{StarCondSL} and the choice of $\chi_0\in \mathbb{T}$, each factor $\chi_{\ol{\mu}}$ of $\chi_0$ extends to $L_{\ol{\mu}}(\wh{V}_{\ol{\mu}})_{\chi_{\ol{\mu}}}/\GenGp{v_{\ol{\mu}}\wh{F}_{\ol{\mu}}}$.
Hence $\chi_{\ol{\mu}}$ extends to some $\phi_{\ol{\mu}}\in \Irr \left( L_{\ol{\mu}}(\wh{V}_{\ol{\mu}})_{\chi_{\ol{\mu}}} \right)$ such that $v_{\ol{\mu}}\wh{F}_{\ol{\mu}}\in \ker (\phi_{\ol{\mu}})$.
Furthermore, these extensions can be taken so that, $\phi_{\ol{\mu}}^{s_{\ol{\mu},\ol{\mu}'}}=\phi_{\ol{\mu}'}$ whenever $\chi_{\ol{\mu}}^{s_{\ol{\mu},\ol{\mu}'}}=\chi_{\ol{\mu}'}$.
Therefore 
\[
\phi_0:=\phi_{\lambda_0}\times \prod\limits_{\ol{\mu}\in\ol{\lambda}'} \phi_{\ol{\mu}}
\]
provides an extension of $\chi_0$ to $L_0(V_0E_0)_{\chi_0}$ with $v\wh{F}\in \ker (\phi_0)$ and $S_{\chi_0}=S_{\phi_0}$.

As $\wt{K}$ is the central product of $L_0$ and $Z_2$, there is a character $\tau \in \Irr (Z_2)$ such that $\wt{\psi}(lz)=\chi_0(l)\tau(z)$.
Furthermore, $\wt{K}(V_0E_0)_{\wt{\psi}}$ is the central product of $Z_2$ and $L_0(V_0E_0)_{\chi_0}$ and therefore $\wt{\phi}:=\phi_0.\tau$ defines an extension of $\wt{\psi}$.
Set $\phi:=\Res_{K(V_0E_0)_{\psi}}^{\wt{K}(V_0E_0)_{\wt{\psi}}}(\wt{\phi})$ which then forms an extension of $\psi$.
Thus $\eta:=\Res_{K(\wh{V}_0)_{\psi}}^{K(V_0E_0)_{\psi}}(\phi)$ provides an extension of $\psi$ to $ K(\wh{V}_0)_{\psi}$ and by construction contains $v\wh{F}$ in its kernel.

It is clear that $S_{\phi}\leq S_{\psi}$.
Therefore let $s\in S_{\psi}$ and consider $\phi^s$.
As $\wt{\psi}^s$ is also an extension of $\psi$, it follows that $\wt{\psi}^s=\chi_0.\tau$ or $\chi_0.(\tau\cdot \det_{Z_2})$.
However $S_{\psi}\leq S_{\chi_0}$ and therefore $\wt{\psi}^s=\chi_0.\tau^s$ and so $\tau^s=\tau$ or $\tau\cdot \det_{Z_2}$.
Hence $\wt{\phi}^s=\phi_0.\tau$ or $\phi_0.(\tau\cdot \det_{Z_2})$ and thus
\[
\phi^s=\Res_{K(V_0E_0)_{\psi}}^{\wt{K}(V_0E_0)_{\wt{\psi}}}(\wt{\phi}^s)=\Res_{K(V_0E_0)_{\psi}}^{\wt{K}(V_0E_0)_{\wt{\psi}}}(\wt{\phi})=\phi.
\]
In particular, $S_{\phi}=S_{\psi}=S_{\eta}$.

By construction, $K=L_0\Zent(K)$ with $S\cap K\leq \Zent(K)$.
Therefore $\Res^K _{S\cap K}(\psi)=m\zeta$ for some $m\in \N$ and $\zeta\in \Irr\left( \Res_{S\cap K}^K(\psi)\right)$.
Furthermore, $\zeta$ extends to a character $\wt{\zeta}\in \Irr (S_{\zeta})$ by Corollary~\ref{ExtK0ToS}.
As
\[
L_0\rtimes \rho(S)=L_{\lambda_0}\times \prod_{f=1}^n \left( L_f\rtimes \rho(S_f) \right),
\]
it follows that $\chi_0$ extends to $L_0\rtimes \rho(S_{\chi_0})$ by 
\cite[Corrollary 10.2]{NavMcKay}.
Hence by Proposition \ref{prop_tool}, $\psi$ extends to $KS_{\psi}$.

Observe that $L\wh{V}_0\cap S\leq K$ and therefore $K\wh{V}_0\cap KS=K(K\wh{V}_0\cap S)=K$.
As $\psi$ extends to $KS_{\psi}$, \cite[Lemma 4.1]{SpSylowTori2} implies that $\eta$ extends to $\left( K(\wh{V}_0)_{\psi}\right) \left(KS\right)_{\eta}=K(\wh{V}_0)_{\psi}S_{\psi}$.
Hence $\psi$ extends to $M_{\psi}$ with $v\wh{F}$ contained in its kernel.
This completes the proof of Theorem~\ref{NEquivExt}.

\subsection{Characters of the quotient $N/L$}
Fix $\lambda\vdash \ul n$ satisfying condition (ii) of Proposition~\ref{PartwStableRootSubsys}, $\wt L:=\wtbL_{\lambda}^{vF}$, $L:={\wt{L}}\cap \bG^{vF}$ and $N=\NNN_{\bG^{vF}}(\bL_\lambda)$.
All the notations from Section~\ref{dSplitLevi} will be used without further reference.
The aim in this section is to study characters of certain subgroups of the relative Weyl group $W:=N/L$.

Let $\chi\in \Irr (L)$ and $\wt{\chi}\in\Irr (\wt L_{\chi}\mid \chi)$.
The quotient $\wt L/L$ is cyclic and therefore $\Res^{\wt L_\chi}_L(\wt{\chi})=\chi$.
Moreover, $\ol{\chi}:=\Ind_{\wt L_{\chi}}^{\wt L}(\wt{\chi})\in \Irr (\wt L)$.
Because $\wt L_{\chi}\lhd N$, we see that $N_{\wt{\chi}}\leq N_{\ol{\chi}}\leq \wt LN_{\wt{\chi}}$.
However all elements in $N$ have determinant one and therefore $N_{\ol{\chi}}=LN_{\wt{\chi}}=N_{\wt{\chi}}$.
Furthermore, if $x\in N_{\chi}$, then $\wt{\chi}^x=\wt{\chi}\beta$ for some $\beta\in \Irr (\wt L_{\chi}/L)$.
However, $\beta$ is a power of the determinant homomorphism and therefore $N_{\wt{\chi}}\lhd N_{\chi}$.

\begin{prop}\label{RelWeylExt}
Let $\chi\in \Irr (L)$, $\wt{\chi}\in \Irr (\wt L_{\chi}\mid \chi)$,
\[
W_{\chi}:=N_{\chi}/L,\indent W_{\wt{\chi}}:=N_{\wt{\chi}}/L,\indent W_{\ol{\chi}}:=N_{\ol{\chi}}/L, \indent K:=\NNN_W(W_{\chi},W_{\wt{\chi}})
\]
and $\eta_0\in \Irr (W_{\wt{\chi}})$.
Then there exists a character $\eta\in \Irr (W_{\chi}\mid \eta_0)$ such that
\begin{asslist}
\item $\{ \eta^w\mid w\in K\}\cap \Irr (W_{\chi}\mid \eta_0)=\{\eta\}$,
\item $\eta$ extends to $K_{\eta}$,
\item $\eta$ has an extension $\wt{\eta}\in \Irr (K_{\eta}\times E)$ with $v\wh{F}\in \ker (\wt{\eta})$.
\end{asslist}
\begin{proof}
For each $f\in \{1,\dots, n\}$ set $W_f=LV_fS_f/L$, $\wh{Z}_f=LV_f/L$ and $\wh{S}_f=LS_f/L$.
Then $\wh{Z}_f\cong (C_{d_0})^{t_f}$, $\wh{S}_f\cong \Symm_{t_f}$ and $W_f=\wh{Z}_f\rtimes \wh{S}_f$.
Moreover $W=\prod\limits_{f=1}^n W_f$ and for $\ol{\chi}_f:=\Res_{\wt L_f}^{\wt L}(\ol{\chi})$ it follows that
\[
W_{\ol{\chi}}=\prod\limits_{f=1}^n (W_f)_{\ol{\chi}_f}.
\]
Therefore $\eta_0\in \Irr (W_{\ol{\chi}})$ can be written as $\eta_0=\prod\limits_{f=1}^n \eta_{0,f}$, where $\eta_{0,f}:=\Res_{(W_f)_{\ol{\chi}_f}}^{(W)_{\ol{\chi}}}(\eta_0)$.

After suitable $V_f$-conjugation, the character $\ol{\chi}_f$ has a stabiliser in $W_{\ol{\lambda}_f}$ with the following description: there exists positive integers $r_f, d_{f,j}$, $a_{f,j}$ with $1\leq j\leq r_f$ and a partition $M_{f,1},\dots,M_{f,r_f}$ of $\{1,\dots, t_f\}$ such that $|M_{f,j}|=a_{f,j}$ and
\[
W_{f,\ol{\chi}_f}=\left\{ \left( (\zeta_1,\dots,\zeta_{t_f}),\sigma \right) \in C_{d_0}\wr \Symm_{t_f}\mid \sigma(M_{i,j})=M_{i,j} \text{ for all } 1\leq j\leq r_i \text{ and } \zeta_k^{d_{i,j}}=1 \text{ for all } k\in M_{i,j} \right\}. 
\]
The group $W_{f,\ol{\chi}_f}$ is isomorphic to a group as considered in \cite[Proposition~5.12]{CabSpCharTypeA}.
Therefore the following constructions from \cite[Proposition~5.12]{CabSpCharTypeA} are taken:

For $\nu_f\in \Irr \left( \Res_{\wh{Z}_{f,\ol{\chi}_f}}^{(W_f)_{\ol{\chi}_f}}(\eta_{0,f})\right)$, there exists an extension $\psi_f\in\Irr \left(\wh Z_f(W_{f,\ol{\chi}_f})_{\nu_f}\right)$ with $(S_{f,\ol{\chi}_f})_{\nu_f}\in \ker (\psi_f)$.
Furthermore there exists another character $\wt{\kappa}_f\in\Irr \left(\wh{Z}_f(W_{f,\ol{\chi}_f})_{\nu_f}\right)$ with $\wh{Z}_f\in \ker (\wt{\kappa}_f)$ so that the character
\[
\wt{\eta}_{0,f}:=\Ind_{\wh Z_{f}(W_{f,\ol{\chi}_f})_{\nu_f}}^{\wh Z_{f}W_{f,\ol{\chi}_f}}(\psi_f \wt{\kappa}_f)
\]
satisfies
\[
\Res_{W_{f,\ol{\chi}_f}}^{\wh Z_{f}W_{f,\ol{\chi}_f}}(\wt{\eta}_{0,f})=\eta_{0,f}.
\]
Moreover $\NNN_{W_f}(W_{f,\ol{\chi}_f})_{\eta_{0,f}}\leq \NNN_{W_f}(\wh Z_fW_{f,\ol{\chi}_f})_{\wt{\eta}_{0,f}}$ and $\wt{\eta}_{0,f}$ has an extension   $\phi_f\in \Irr \left( \NNN_{W_f}(\wh Z_fW_{f,\ol{\chi}_f})_{\wt{\eta}_{0,f}}\right)$.

Define $\wt{\eta}_0:=\prod\limits_{f=1}^n \wt{\eta}_{0,f}$ and $\wh Z:=\prod\limits_{f=1}^n \wh Z_f$ so that
\[
\Res_{W_{\ol{\chi}}}^{\wh ZW_{\ol{\chi}}}(\wt{\eta}_0)=\eta_0.
\]
Then
\[
\NNN_W(W_{\ol{\chi}})_{\eta_0}=\prod\limits_{f=1}^n \NNN_{W_f}(W_{f,\ol{\chi}_f})_{\eta_{0,f}}\leq \prod\limits_{f=1}^n \NNN_{W_f}(\wh Z_fW_{f,\ol{\chi}_f})_{\wt{\eta}_{0,f}}= \NNN_W(\wh ZW_{\ol{\chi}})_{\wt{\eta}_0}
\]
and so $\phi:=\prod\limits_{f=1}^n \phi_f\in \Irr \left( \NNN_W(\wh ZW_{\ol{\chi}})_{\wt{\eta}_0}\right)$ is an extension of $\wt{\eta}_0$.
Hence as $W_{\chi}\leq \NNN_W(W_{\ol{\chi}})$, the character
\[
\wt{\eta}_0:=\Res_{(W_{\chi})_{\eta_0}}^{\NNN_W(\wh ZW_{\ol{\chi}})_{\wt{\eta}_0}}(\phi)
\]
is an extension of $\eta_0$ and 
\[
\eta:=\Ind_{(W_{\chi})_{\eta_0}}^{W_{\chi}}(\wt{\eta}_0)\in \Irr (W_{\chi}\mid \eta_0).
\]

If $w\in K$ and $\eta^w\in \Irr (W_{\ol{\chi}}\mid \eta_0)$ then $w\in W_{\chi}K_{\eta_0}$ and hence it can be assumed that $w\in K_{\eta_0}$.
Furthermore, $K_{\eta_0}\leq \NNN_W(\wh ZW_{\ol{\chi}})_{\ol{\eta}_0}$ and so $\phi^w=\phi$.
However this implies that $\wt{\eta}_0^w=\wt{\eta}_0$ as $\wt{\eta}_0$ is the restriction of $\phi$, and thus $\eta^w=\eta$ as $\eta$ arises as the induced character which is fixed by elements of $K_{\eta_0}$, proving the first statement.
Moreover, the observation $K_{\eta}=W_{\chi}K_{\eta_0}$ yields that
\[
\wt{\eta}:=\Ind_{K_{\eta_0}}^{K_{\eta}}\left( \Res_{K_{\eta_0}}^{\NNN_W(\wh ZW_{\ol{\chi}})_{\wt{\eta}_0}} (\phi) \right)\in \Irr (K_{\eta})
\]
is an extension of $\eta$ proving the second statement.
The final property is the same as the property given in \cite[Proposition~5.12]{CabSpCharTypeA} and the proof is the same.
\end{proof}
\end{prop}

The explicit description allows us to see that the relative Weyl groups $ W_\chi$ satisfy the McKay Conjecture. This is applied in the proof of Theorem \ref{thm13a}.
\begin{prop}\label{propMcKayrelWeyl}
The McKay Conjecture holds for $W_\chi$ from \ref{RelWeylExt}, whenever $\chi\in\Irr(L)$.
\end{prop} 
\begin{proof}
Recall that $W=\wh Z\rtimes \wh S$ and $W_{\wt \chi}=Z\rtimes S$ for an extension $\wt \chi$ of $\chi$ to $\wt L_\chi$.
By construction $W_\chi / W_{\wt \chi}$ is isomorphic to a subgroup of $\wt L_\chi /L$, hence cyclic.
The group $Z\lhd W_\chi$ is abelian and by the constructions given in the previous proof, every character of $Z$ extends to its stabilizer in $ W_\chi$. Let $\mu\in \Irr(Z)$. Then $S_\mu$ is a direct product of symmetric groups and $(W_\chi)_\mu /(ZS_\mu)$ is cyclic.
According to \cite[Lemma 12.4]{SpSylowTori1}, $(W_\chi)_\mu/Z$ satisfies the McKay conjecture.
This allows us to apply \cite[Lemma 12.1]{SpSylowTori1} and obtain that the McKay conjecture holds for $W_\chi$.
\end{proof}

\begin{cor}\label{ExtNToNE}
	\begin{thmlist}
\item For every $\chi\in\Irr (\wt N)$ there exists some $\chi_0\in\Irr (N\mid \chi)$ such that 
\begin{asslist}
\item $(\wt N\wh N)_{\chi_0}=\wt N_{\chi_0}\wh N_{\chi_0}$, and
\item $\chi_0$ has an extension $\wt{\chi}_0$ to $\wh N_{\chi_0}$ with $v\wh F\in \ker (\wt{\chi})$. 
\end{asslist}
\item \label{extmapwt}
There exists some $\wh N$-equivariant extension map $\wt \Lambda$ with respect to $\wt L \lhd \wt N$ that is compatible with $\Irr(\wh N/N)$.
\end{thmlist}
\end{cor}
\begin{proof}
This follows by combining Theorem~\ref{MainThmReq} with the verification of the required conditions in Theorem~\ref{NEquivExt}, Proposition~\ref{RelWeylExt} and Corollary~\ref{IntCond}.
\end{proof}
\label{Results}
We are finally able to verify Theorem~\ref{StarConddSpLevi}.

\begin{proof}[Proof of Theorem~\ref{StarConddSpLevi}]
Using the proof of \cite[Proposition 5.3]{CabSpCharTypeA} there exists an isomorphism 
$\epsilon: \wt{\bG}^F\rtimes E_0 \rightarrow \Cent_{\wt{\bG}^{F_0^{em}E}}(v\wh F)/\langle v\wh F\rangle$.
Let $\bL'$, $N'$ and $\wt N'$ be the images of $\bL_0$, $N_0$ and $\wt N_0$ respectively under this isomorphism.
By the proof of \cite[Proposition 5.3]{CabSpCharTypeA}, Corollary~\ref{ExtNToNE} implies the statement.
\end{proof}

\section{Block-theoretic considerations}
The aim in this section is to study the $\ell$-blocks of $\SL_n(\epsilon q)$ via relating the normalizer of a defect subgroup with normalizers of $d$-split Levi subgroups ($d$ is the multiplicative order of $q$ mod $\ell$)
and constructing a bijection between height zero characters. 

Here we rely on Brou\'e-Malle-Michel's theory of generic Sylow $d$-tori and generic blocks for finite reductive groups from \cite{BMM}. Using this language, Cabanes-Enguehard parametrized the $\ell$-blocks of finite reductive groups $\bG^F$ with $d$-cuspidal pairs $(\bL,\zeta)$ where $\bL$ is a $d$-split Levi subgroup and $\zeta$ a so-called $d$-cuspidal character of $\bL^F$, see \cite[4.1]{CabEngAdv}. The $\ell$-block $B$ of finite reductive group $(\bG,F)$ associated with $(\bL,\zeta)$, denoted by $b_{\GF}(\bL,\zeta)$, is the $\ell$-block that contains all irreducible constituents of $\RDL_{\bL}^{\bG}(\zeta)$.

In order to prove Theorem \ref{thmA}, via an application of the criterion introduced in Theorem \ref{NewIndAmCond}, it is necessary to prove that the normalizers of the $d$-split Levi subgroups from Theorem \ref{StarConddSpLevi} can be chosen as the group denoted by $M$ in Theorem \ref{NewIndAmCond}. We prove this in the case where $\ell\nmid 3q(q-\epsilon)$. 

In a second step we construct a bijection $\widetilde \Omega_{\wt B}$ with the properties required in \ref{thm24ii}. This bijection has obvious similarities with the one constructed for the inductive McKay condition, which was first developed in \cite{MalleHeight0} and then later in \cite{CabSpCharTypeA}. As main ingredients it uses the so-called $d$-Harish-Chandra theory, the Jordan decomposition of characters and extension maps for $d$-split Levi normalizers. This bijection has to be transferred to the new context where also blocks and the height of the characters have to be taken into account.

\subsection{Normalizers of defect groups and height zero characters}

Recall that for a odd prime $\ell$ and a prime power $q$ with $\ell\nmid q$ we denote by $d_\ell(q)$ the order of $q$ in $(\Z/\ell\Z)^\times$.
For $(\bH ,F)$ a reductive group defined over a finite field, denote by $\cE_{\ell'}(\bH^F)$ the union of Lusztig series associated to semi-simple $\ell'$-elements of $\bH^*{}^F$.

\begin{thm}\label{71b}
	Let $\wt\bG=\GL_{n}(\overline{\FF}_q)$ , $F:\wt\bG\to \wt\bG$ a Frobenius endomorphism defining an $\FF_q$-structure, $\epsilon\in\{\pm 1\}$ with $\wGF=\GL_n(\epsilon q)$ and $\ell$ a prime with $\ell\nmid 3q(q-\epsilon)$. Let $\wt B\in \Bl( \wt \bG^F)$ and set $d:=d_\ell(q)$. Let $(\bL,\zeta)$ be a $d$-cuspidal pair of $(\wbG,F)$ associated to $\wt B$ as in \cite[4.1]{CabEngAdv}. Let $\bS$ be the Sylow $\Phi_d$-subtorus of $\Zent^\circ(\bL)$. 
		Then there exists some defect group $D$ of $\wt B$ such that $\NNN_{[\wbG,\wbG]^F}(\bS)$ is $\Aut([\wbG,\wbG]^F)_{\wt B,D}$-stable and $\NNN_{\wGF}(D)\leq \NNN_{\wGF}(\bS)$. Moreover $\Cent_{\wGF}(D)\leq \Cent_{\wGF}(\bS)=\bL^F$. 
\end{thm}
\begin{proof} We have $\wt B=b_{\wGF}(\bL,\zeta)$ in the notation of \cite[2.6]{CabEngAdv}, i.e., $\bL$ is a $d$-split Levi subgroup, $\zeta\in\cE_{\ell'}(\bL^F)$ is $d$-cuspidal and all constituents of $\RDL_{\bL}^{\wbG}(\zeta)$ are contained in $\wt B$. 

We observe that by our assumption on $\ell$ the groups $\wt\bG_{\bf a}$ and $\wt\bG_{\bf b}$ defined in the paragraph after Proposition 3.3 of \cite{CabEngAdv} satisfy $\wt\bG_{\bf a}=\Zent (\wt\bG)$ and $\wt\bG_{\bf b} =\bG\cong \SL_n(\overline{\F}_q)$. Further note that $\bL$ coincides with the Levi subgroup $\bK$ defined in \cite[3.2 and 3.4]{CabEngAdv}. 

Let $\bM$ be the group defined in the paragraph before \cite[4.4]{CabEngAdv}. Then $\bS$ is also the Sylow $\Phi_d$-subtorus of $\Zent^\circ(\bM)$ thanks to \cite[4.4.(iii)]{CabEngAdv}. Denote $Z:= \Zent(\bM)_\ell^F$. Then $\bM=\Cent^\circ_\wbG(Z)$ according to \cite[Lemma~ 4.8]{CabEngAdv}. 

By \cite[Lemma~4.16]{CabEngAdv}, $Z$ is a characteristic subgroup of a defect group $D$ of $\wt B$ and hence $\NNN_{\wGF}(D)\subseteq \NNN_{\wbG^F}(Z)$. 
We also have $\NNN_{\wbG^F}(Z)=\NNN_{\wbG^F}(\bM)=\NNN_{\wbG^F}(\Zent^\circ(\bM))\leq\NNN_{\wbG^F}(\bS)$. 

Denote $\bG=[\wbG,\wbG]$. We must ensure that every $\phi_0\in\Aut(\GF)_{\wt B,D}$ stabilizes $\NNN_{\bG^F}(\bL)$. Such $\phi_0$ is induced by a bijective endomorphism $\phi$ of $\wt\bG$ commuting with $F$. The equality $\bM=\Cent^\circ_{\wt\bG}(Z)$ implies $\phi(\bM)=\bM$. Similarly $\phi$ stabilizes $\bL$ since $\bL\cap \bG_b$ is the smallest $d$-split Levi subgroup containing $\bM$, see \cite[proof of Lemma~4.4]{CabEngAdv}. 

Since the centralizer of any semi-simple element is connected in $\wbG$, $\Cent_{\wbG}(Z)$ is connected and we may write $\Cent_\wbG(D)^F\leq \Cent_{\wbG}(Z)^F =\Cent_{\wbG}^\circ(Z)^F =\bM^F\leq \bL^F= \Cent_{\wbG}(\bS)^F$ by the definition of $\bS$. 
\end{proof}

To construct the bijection required by Theorem~\ref{NewIndAmCond} the following description of height $0$ characters of unipotent blocks is needed. The general proof below was communicated to us by Marc Cabanes. The particular case of type A could be treated with a simpler proof.

\begin{thm}[Height zero and series]\label{Conjunip} Let $(\bG , F)$ be a reductive group defined over a field of cardinality $q$, $\ell$ an odd prime good for $\bG$, not dividing $q$, and $\not= 3$ if $\bG^F$ is of type $^3D_4$, $d$ the multiplicative order of $q$ mod $\ell$. 
	Let $B$ be a unipotent $\ell$-block of $\bG^F$ defined by a $d$-cuspidal pair $(\bK, \zeta)$ as in \cite[4.4]{CaEn94}. Then $$\Irr_0(B)\subseteq \cup_{t}\ \cE(\bG^F ,t)$$ where $t$ ranges over $ \Zent(\bK^*)^F_\ell$. 
	Each such $t\in \Zent(\bK^*)^F_\ell$ satisfies $\ell\nmid |\NNN_{\bG}(\bK)^F: \NNN_{\bG(t)}(\bK)^F|$, where $\bG(t)$ denotes a Levi subgroup of $\bG$ dual to $\Cent_{{\bG}^*}(t)$ with $\bK\leq\bG(t)$. 

\end{thm}

\begin{proof} 
	A character in $\Irr_0(B)$ is of the form $\chi:=\pm \RDL ^\bG_{\bG(t)}(\hat{t}\mu_t)$ where $t\in\bG^*{}^F_\ell$, $\bG(t)$ is an $F$-stable Levi subgroup of $\bG$ in duality with the $F$-stable Levi subgroup $\Cent_{\bG^*}^\circ (t)$, $\hat{t}$ is a linear character of $\bG(t)^F$ defined by duality and $\mu_t\in\cE(\bG(t)^F,1)$ is a component of R$^{\bG (t)}_{\bK_t}\zeta_t$ where $({\bK_t},\zeta_t)$ is a unipotent $d$-cuspidal pair in $\bG(t)$ (see \cite[ 4.4(iii)]{CaEn94}) and $[\bK ,\bK]=[\bK_t ,\bK_t]$. By
	\cite[4.4(ii)]{CaEn94} any Sylow $\ell$-subgroup of $\Cent_\bG^\circ([\bK ,\bK])^F$ is a defect group of $B$. 
	Similarly $\mu_t$ belongs to an $\ell$-block of $\bG(t)^F$ such that any Sylow $\ell$-subgroup of $\Cent_{\bG(t)}^\circ([\bK ,\bK])^F$ is a defect group of this block.
	
	If $\chi$ has height zero, then 
	$\chi(1)_\ell =|\bG^F:\bG(t)^F|_\ell \cdot \mu_t(1)_\ell=|\bG^F:\Cent_\bG^\circ([\bK ,\bK])^F|_\ell$. On the other hand $\mu_t(1)_\ell =\ell^{h(\mu_t)} \, |\bG(t)^F:\Cent_{\bG(t)}^\circ([\bK ,\bK])^F|_\ell$, 
	where $0\leq h(\mu_t)$ denotes the height of $\mu_t$, by what has been said about $\bl(\mu_t)\in\Bl(\bG(t)^F)$. Thus $\ell^{h(\mu_t)}|\bG^F:\Cent_{\bG(t)}^\circ([\bK ,\bK])^F|_\ell =|\bG^F:\Cent_{\bG}^\circ([\bK ,\bK])^F|_\ell$ and therefore $$h(\mu_t)=0 \text{ and } |\Cent_{\bG}^\circ([\bK ,\bK])^F:\Cent_{\bG(t)}^\circ([\bK ,\bK])^F|_\ell =1.$$
	
	We use the second equality. As duality preserves the order of groups of rational points, one has \[|\Cent_{\bG^*}^\circ([\bK^*,\bK^*])^F: \Cent_{\bG^*}^\circ([\bK^* ,\bK^*],t)^F|_\ell =1\] with $\bK^*$ in duality with $\bK$ and $t\in \Cent_{\bG^*}([\bK^* ,\bK^*])^F_\ell$. This means that $t$ centralizes a Sylow $\ell$-subgroup of $C:=\Cent_{\bG^*}^\circ([\bK^* ,\bK^*])^F$. So it has a $C$-conjugate that centralizes $\Zent^\circ(\bK^*)^F_\ell$. But $\Cent_{\bG^*{}^F}(\Zent^\circ(\bK^*)^F_\ell)= \bK^*{}^F$ by \cite[ 3.3(ii)]{CaEn94} and therefore $\Cent_C(\Zent^\circ(\bK^*)^F_\ell)\leq \Cent_{\bK^*}([\bK^*,\bK^*])^F=\Zent(\bK^*)^F$. This gives that $t$ is actually conjugate to an element of $\Zent(\bK^*)^F_\ell$ in $\bG^*{}^F$. 
	
We have to show now that $ |\NNN_{\bG}(\bK)^F: \NNN_{\bG(t)}(\bK)^F|_\ell =1$. Letting $g\in \NNN_{\bG}(\bK)^F_\ell$ we must show that $g\in \bG(t)$. By \cite[6]{CabEnFus}, we have $g\in (\bK \Cent_{\bG}^\circ([\bK ,\bK]))^F_\ell$. Thanks to \cite[ 3.3(ii)]{CaEn94} telling us that $\Zent(\bK)^F_\ell =\Zent^\circ(\bK)^F_\ell $ and Lang-Steinberg's theorem, we have $(\bK \Cent_{\bG}^\circ([\bK ,\bK]))^F_\ell =\bK^F_\ell \Cent_{\bG}^\circ([\bK ,\bK])^F_\ell$. But we have seen before that $\Cent_{\bG}^\circ([\bK ,\bK])^F_\ell\leq \bG(t)$, so indeed $g\in \bK\bG(t)=\bG(t)$.
\end{proof}

Note that the above proof simplifies when $\bG$ is a general linear group, as in that case $\bG =\bG^*$ and all centralizers are connected.

\subsection{Bijections between certain characters of $\wGF$ and $\NNN_\wGF(\bS)$}
Recall that for a given character $\chi$ we use $\bl(\chi)$ to denote the $\ell$-block $\chi$ belongs to.

In the next step we show that a bijection as required in \ref{thm24ii} exists for a block $\wt B\in \Bl(\GL_n(\epsilon q))$. Note that although Brou\'e's conjecture is known for blocks with abelian defect of $\GL_n(q)$, the Alperin-McKay conjecture hasn't been proven for $\SL_n(q)$. 

Let $\wbG$, $F:\wbG\to \wbG$ such that $\wGF\cong\GL_n(\epsilon q)$. Note that $\Irr(\wGF/[\wbG ,\wbG]^F)$ acts on $\Irr(\wGF)$ by multiplication and thereby induces an action on $\Bl(\wGF)$. Note that any $\Irr(\wGF/[\wbG ,\wbG]^F)$-orbit coincides with $\Bl(\wGF\mid B)$, the set of blocks of $\wGF$ covering $B$ for some $B\in\Bl([\wbG ,\wbG]^F)$. 

\subsubsection{A parameter set and two character sets associated to a $\Phi_d$-torus} Let
 $d\geq 1$ and $\bS$ a $\Phi_d$-torus of $(\wbG,F)$. Denote $\bL:=\Cent_{\wbG}(\bS)$. For any $F$-stable torus $\bT$ of $\wbG$ we denote by $\bT_{\Phi_d}$ the Sylow $\Phi_d$-torus of $(\bT, F)$.

For $s\in (\wbG^*)^F$ semi-simple, let $\wbG(s)$ be an $F$-stable Levi subgroup of $\wbG$ dual to $\Cent_{\wbG^*}(s)$ and $\wh s$ the linear character of $\wbG(s)^F$ associated to $s$ by duality.
If $\bL'$ is an $F$-stable Levi subgroup of $\wbG$, $\bK$ an $F$-stable Levi subgroup of $\bL'$ and $\kappa\in\Irr(\bK^F)$, then 
$W_{\bL'}(\bK,\kappa)^F$ is defined as $\NNN_{\bL'^F}(\bK)_\kappa/\bK^F$. 
If additionally $(\bK,\kappa)$ is a unipotent $d$-cuspidal pair of $(\bL',F)$ we denote by $\cE(\bL'{}^F,(\bK,\kappa))$ the set of constituents of $\RDL_{\bK}^{\bL'}(\kappa)$ and there is a bijection $$\Irr(W_{\bL'}(\bK,\kappa)^F)\to\cE(\bL'{}^F,(\bK,\kappa))$$ according to \cite[3.2(2)]{BMM} which we denote by $$\eta\mapsto\RDL_{\bK}^{\bL'}(\kappa)_\eta .$$

For any element $x$ of a finite group we denote by $x_{\ell'}$ an element of $\langle x \rangle $ whose order is not divisible by $\ell$ and for which $xx_{\ell'}^{-1}$ is an $\ell$-element. We define $\cP_\bS$ as the set of triples $ (s,\kappa, \eta)$ where 
\begin{itemize}
	\item $s\in \wt\bG^*$ is a semi-simple element
	with $\bS\leq \wbG(s)$ and 
	$\Cent_{\wbG(s_{\ell'})}(\bS)\leq \wbG(s)$,
	\item $\kappa$ is a unipotent $d$-cuspidal 
	character of $\bK^F$,
	where $\bK:=\Cent_{\wbG(s_{\ell'})}(\bS)$ with 
	$\bS=\Ze(\bK)_{\Phi_d}$, and 
	\item $\eta\in\Irr(W_{\wt\bG(s)}(\bK,\kappa)^F)$. 
\end{itemize}
The group $\wt N:=\NNN_\wGF(\bS)$ acts via conjugation on $\cP_\bS$ and we write $\overline \cP_\bS$ for the set of $\wt N$-orbits in $\cP_{\bS}$. 

Let $\Upsilon^\circ: \cP_\bS \rightarrow \Irr(\wGF)$ be given by $(s,\kappa, \eta)\mapsto \epsilon_{\wbG }\epsilon_{\wbG(s)}\RDL_{\wbG(s)}^\wbG( \widehat s\RDL_{\bK}^{\wbG(s)}(\kappa)_{\eta})$, where $\epsilon_{\wbG }$, $\epsilon_{\wbG(s)}$ are signs (see \cite[8.27]{CabEnRedGp}). Denote \[\calG_\bS:=\Upsilon^\circ(\cP_\bS).\] 

\begin{lm}\label{lem_deg}
	Let $(s,\kappa,\eta)\in \cP_{\bS}$. Then $\Upsilon^\circ(s,\kappa,\eta)(1)_\ell=|\wGF:\NNN_{\wbG(s)^F}(\bK)| _\ell \,\eta(1)_\ell \, \kappa(1)_\ell$.
\end{lm}
\begin{proof}
In this case Jordan decomposition coincides with Deligne-Lusztig induction and 
$\RDL_{\wbG(s)}^\wbG( \widehat s\RDL_{\bK}^{\wbG(s)}(\kappa)_{\eta})(1)=\epsilon_{\wbG }\epsilon_{\wbG(s)} 
|\wGF:\wbG(s)^F|_{p'} \RDL_{\bK}^{\wbG(s)}(\kappa)_{\eta}(1)$. 
The degree of $\RDL_{\bK}^{\wbG(s)}(\kappa)_{\eta}$ is given in 
Theorem~4.2 of \cite{MalleHeight0}. Let $D_\eta\in\Q(X)$ be as in 
\cite[4.2]{MalleHeight0} a rational function with zeros and poles only at roots of unity and $0$ such that $D_\eta(\zeta_d)=\eta(1)/|W_{\wbG(s)^F}(\bK)|$ for a certain primitive $d$th root $\zeta_d$ of unity and 
\[\RDL_{\bK}^{\wbG(s)}(\kappa)_{\eta}(1) 
=|\wbG(s)^F:{\bK}^F|_{p'} 
\kappa(1)D_\eta(q) .\]

Assume that the defect group of the block is abelian and hence contained in $\bK^F$. In that case we see that the $\ell$-evaluation of $D_\eta(q)$ defined as in \S 6 of \cite{MalleHeight0} is $0$. 
Recall $D_\eta\in \Q(X)$ is a ratio of products of cyclotomic polynomials and possibly powers of $X$.
Since $D_\eta$ only depends on $d$, the order of $\ell$ mod $q$, and in all those cases the above equality holds. This implies that when we can write $D_\eta$ as quotient of two polynomials that are not divisible by $\Phi_d$. Accordingly for the computation of the $\ell$-valuation of $D_\eta(q)$ we can apply Corollary 6.3 of \cite{MalleHeight0}. This gives that the $\ell$-evaluation of $D_\eta(q)$ and $D_\eta(\zeta_d)$ coincide. Recall $D_\eta(\zeta_d)=\frac{\eta(1)}{|W_{\wt\bG(s)^F}(\bK,\kappa)|}$. 
Hence
\[\RDL_{\bK}^{\wbG(s)}(\kappa)_{\eta}(1)_\ell
=|\wbG(s)^F: \NNN_{\wbG(s)^F} (\bK)|_{\ell} \, \kappa (1)_\ell \,\eta(1)_\ell .\]
This implies the stated formula in the general case. 
\end{proof}

Let $\wt  L:=\Cent_\wGF(\bS)$ and $\wt \Lambda$ be the $\wt N$-equivariant extension map for $\wt L\lhd \wt N$ from Proposition \ref{extmapwt}. 
Let $\Upsilon'^\circ: \cP_\bS \rightarrow \Irr(\wt N)$ be given by $(s,\kappa, \eta)\mapsto 
\Ind^\wN_{\wN_{\rho}}(\wt \Lambda(\rho)\eta)
$ for $\rho :=\epsilon_{\bL ,\bK}\RDL_\bK^{\bL}(\widehat s \kappa)$. Denote 
\[\calN_\bS:=\Upsilon'^\circ(\cP_\bS). \]
This uses that $\wt N_\rho/\wt L$ and $W_{\wtbL\bG(s)}(\bK,\kappa)^F$ are canonical isomorphic by \cite[3.3]{CabSpMZ}.

\begin{lm}\label{lem_deg_loc}
	Let $(s,\kappa,\eta)\in \cP_{\bS}$. Then $\Upsilon'^\circ(s,\kappa,\eta)(1)_\ell=\frac{|\wt N:\bK^F|_\ell}{|W_{\wbG(s)^F}(\bK,\kappa)|_\ell}\, \kappa(1)_\ell \,\eta(1)_\ell$.
\end{lm}
\begin{proof}
	We see from the definition $\Upsilon'^\circ(s,\kappa,\eta)(1)_\ell=|\wt N:\wt N_{\rho}| _\ell\, \,\eta(1)_\ell\, |\bL^F:\bK^F|_\ell \, \kappa(1)_\ell$. Moreover $\wN_{\rho}/\bL^F$ and $W_{\wbG(s)^F}(\bK,\kappa)$ are isomorphic. 
This leads to the stated equation. 	
\end{proof}

We abbreviate $\bG:=[\wbG ,\wbG]=\SL_n(\overline \FF_q)$.
\begin{prop}\label{prop72neu}
	Let $\ell$ be a prime with $\ell\nmid q(q-\epsilon)$ and $d:=d_\ell(q)$. 
	Let $\bS$ be a $\Phi_d$-torus of $(\wbG,F)$,
	$\wt N:=\NNN_{\wG}(\bS)$, and $\calG_\bS$ and $\calN_\bS$ be character sets associated to $\bS$ as above.	

	Then there exists a $(\wGF E)_{\bS}$-equivariant bijection 
	\[\wt \Omega_\bS: \calG_\bS\longrightarrow \calN_{\bS}\] 
	with
	\begin{itemize}
		\item $\wt \Omega_\bS(\calG_\bS \cap \Irr(\wt G\mid\nu))\subseteq \Irr(\wt N\mid \nu)$ for every $\nu\in\Irr(\Ze(\wGF))$,
		\item 
		$\wt \Omega_\bS(\chi\mu)= \wt \Omega_\bS(\chi) \Res^{\wGF}_{\wt {N}}(\mu)$ for every $ \chi\in\calG_\bS$ and $ \mu\in\Irr(\wGF/\GF) $, and
		\item $\bl(\chi)=\bl(\wt\Omega_\bS(\chi))^\wG$ for every $\chi\in \calG_\bS$, if $\ell\geq 5$.
	\end{itemize}
\end{prop}
\begin{proof}
The statement is proven by first defining an equivalence relation on $\cP_\bS$ and proving that the maps $\Upsilon^\circ$ and $\Upsilon'^\circ$ induce well-defined injective maps on the set $\overline\cP_\bS$ of equivalence classes in $\cP_{\bS}$. In a second step we then see that the bijection obtained has the required properties. 

Recall $\wt N$ acts by conjugation on $\cP_\bS$ inducing an equivalence relation. We denote by $\overline \cP_\bS$ the set of equivalence classes in $\cP_\bS$.

Via $d$-Harish-Chandra theory and Deligne-Lusztig induction, we see that triples lying in the same $\wt N$-orbit correspond to the same character of $\wG$ by the equivariance properties from \cite[3.1 and 3.4]{CabSpMZ}.
(Note that Theorem~3.4 of \cite{CabSpMZ} states the equivariance of the $d$-Harish-Chandra theory only in the case of minimal $d$-split Levi subgroups, but the proof applies also in the general case.)
Hence $\Upsilon^\circ$ induces a well-defined map $$\Upsilon: \overline \cP_{\bS} \rightarrow \calG_\bS\ .$$ 
Assume for a given character $\chi\in \calG_\bS$ that $\Upsilon^\circ (s,\kappa,\eta)= \Upsilon^\circ (s',\kappa',\eta')=\chi$. We see that $s$ and $s'$ have to be $(\bG^*)^F$-conjugate by the disjointness of Lusztig series. Let $g\in (\wt\bG^*)^F$ with $s'=s^g$. 
According to \cite[3.2(1)]{BMM} the characters $\kappa^g\in \Irr(\Cent_{\wt\bG(s)^F}(\bS)^g)$ and $\kappa'\in \Irr(\Cent_{\wt\bG(s')^F}(\bS))$ are $\wt\bG(s')^F$-conjugate, i.e. 
$\kappa^{gh}=\kappa'$ for some $h\in\wt\bG(s')^F$ and $\eta^{gh}=\eta'$. Since $\Ze(\Cent_{\wt\bG(s')^F}(\bS))_{\Phi_d}=\bS$, we see $gh \in \wt N$. Hence $\Upsilon$ is injective and hence bijective. 

Recall that $\wt \Lambda$ is $\wt N$-equivariant. Therefore $\Upsilon'^ \circ$ induces a bijection 
\[\Upsilon':\overline \calP_\bS \longrightarrow \calN_\bS \]
by Clifford theory and the equivariance of Deligne-Lusztig induction. 

We now study the bijection \[\Omega_{\bS}:\calG_\bS\longrightarrow \calN_{\bS}\text{ given by }\Upsilon'\circ \Upsilon^{-1}.\] 
The considerations in \cite[\S 6]{CabSpCharTypeA} prove that $\Omega_{\bS}$ is $(\wbG^F E)_\bS$-equivariant and satisfies
\begin{itemize}
	\item $\Omega_{\bS}(\calG_\bS \cap \Irr(\wt G\mid\nu))\subseteq \Irr(\wt N\mid \nu)$ for every $\nu\in\Irr(\Ze(\wbG)^F)$
	\item $\Omega_{\bS}$ is compatible with the action of $\Irr(\wGF/\GF)$ by multiplication, i.e.
	\[{\Omega_{\bS}}(\chi\mu)= \Omega_{\bS}(\chi) \Res^{\wGF}_{\wt {N}}(\mu)\text{ for every } \chi\in\calG_\bS 
	\text{ and } \mu\in\Irr(\wGF/\GF).\]
\end{itemize}

It remains to prove $\bl(\chi)=\bl(\Omega_\bS(\chi))^\wG$ for every $\chi\in \Upsilon(\overline \cP_\bS)$. 
Let $(s,\kappa,\eta)\in\cP_{\bS}$ and $\chi=\Upsilon^\circ(s,\kappa,\eta)$. 
Then we see that $\chi$ is a constituent of $\RDL_{\wbG(s)}^\wbG(\widehat {s} \RDL_{\bK}^{\wbG(s)}(\kappa)) $. 
The constituents of $\RDL_{\wbG(s)}^\wbG(\widehat {s} \RDL_{\bK}^{\wbG(s)}(\kappa)) $ lie all in the same block and similarly the constituents of $\RDL_{\wbG(s_0)}^\wbG(\widehat {s_0} \RDL_{\bK}^{\wbG(s_0)}(\kappa))= \RDL_{\bK}^\wbG(\widehat {s_0} \kappa)$ belong to the same block where $s_0:=s_{\ell'}$. To see that, it suffices to combine \cite[22.9.(i)]{CabEnRedGp} and Bonnaf\'e-Rouquier's theorem \cite[10.1]{CabEnRedGp}.

In order to compute the block $\bl(\Upsilon'^\circ(s,\kappa,\eta))^\wGF$ we recall that $\bS^F_\ell$ is a normal $\ell$-subgroup of $\wt N$ and hence the defect group of the block $\bl(\Upsilon'^\circ(s,\kappa,\eta))$ contains $\bS^F_\ell$. By Theorem \ref{71b}, 
$\bL^F\geq \Cent_{\wG^F}(\bS^F_\ell)$. Then 
\[\bl(\Ind_{\wt N_{\RDL_{\bK}^\bL (\widehat s\kappa)}}^{\wt N}
(\Lambda(\RDL_{\bK}^\bL (\widehat s\kappa)) \eta))=\bl(\RDL_{\bK}^\bL (\widehat s\kappa))^{\wt N}\]
according to \cite[9.8]{NavBl}.
By the definition of $s_0$ we see that $\bl(\RDL_{\bK}^\bL (\widehat s\kappa))=\bl(\RDL_{\bK}^\bL (\widehat s_0\kappa))$. This altogether implies that 
\[\bl(\Upsilon'(s,\kappa,\eta))^\wG=
\bl(\RDL_{\bK}^\bL (\widehat s_0\kappa))^\wG.\]
According to \cite[2.5]{CabEngAdv} one knows $\bl(\RDL_{\bK}^\bL (\widehat s_0\kappa))^\wG=\bl(\RDL_{\bK}^\wbG (\widehat s_0\kappa))$.
This shows $\bl(\chi)=\bl(\Omega_\bS(\chi))^\wG$.
\end{proof}

In the next step we obtain a bijection between the height zero characters of Brauer corresponding blocks satisfying the assumption \ref{thm24ii}.
\begin{thm}\label{thm72}\label{cor74a} 
	Assume the situation of \Cref{prop72neu}. 
Let $\bL:=\Cent_{\wbG}(\bS)$, $\zeta\in \cE_{\ell '}(\bL^F)$ be $d$-cuspidal, $\wt B_0=b_{\wGF}(\bL,\zeta)$ the $\ell$-block containing all components of $\RDL_{\bL}^\wbG\zeta$, and $\wt B$ be the sum of $\Irr(\wGF/\GF)$-orbit in $\Bl(\wGF)$ containing $\wt B_0$.
Let $\wt b$ be the sum of $c\in \Bl(\wt N)$ with the property that $c^{\wbG^F}$ is in $\wt B_0$. 

Then there exists a $(\wbG^F E)_{\wt B,\wt N}$-equivariant bijection 
\[\wt\Omega_{\wt B}: \Irr_0(\wt B)\longrightarrow \Irr_0(\wt b)\] 
with
\begin{itemize}
	\item $\wt\Omega_{\wt B}(\Irr_0(\wt B)\cap \Irr(\wGF\mid\nu))\subseteq \Irr(\wt N\mid \nu)$ for every $\nu\in\Irr(\Ze(\wbG^F))$,
	\item $\wt \Omega_{\wt B}$ is compatible with the action of $\Irr(\wbG^F/\GF)$ by multiplication, i.e.,
	\[\wt \Omega_{\wt B}(\chi\mu)= \wt \Omega_{\wt B}(\chi) \Res^{\wGF}_{\wt {N}}(\mu)\text{ for every } \mu\in\Irr(\wbG^F/\GF) \text{ and } \chi\in\Irr(\wt {B}).\]
\end{itemize}
\end{thm} 
\begin{proof}
We have to show $\Irr_0(\wt B)\subseteq \calG_\bS$ and 
$\wt \Omega_\bS(\Irr_0(\wt B))=\Irr_0(\wt b)$ for the bijection $\wt \Omega_\bS$ from Proposition \ref{prop72neu}.

Let $\chi_0\in\Irr(\wt B)\cap \cE_{\ell'}(\wG)$, where $\cE_{\ell'}(\wG)$ is the union of Lusztig series of $\wt\bG^F$ associated to semi-simple $\ell'$-elements, $\chi_0= \Upsilon^\circ(s_0,\kappa,\eta)$ and $\bK$ the Levi subgroup of $\wbG$ such that $\kappa$ is a character of $\bK^F$. 
For any $\chi\in\Irr(\wt B)$ there exists some semi-simple $s\in\wt\bG^*$ with $s_{\ell'}=s_0$ and $\chi\in\cE(\wbG^F,s)$ (\cite[9.12.(i)]{CabEnRedGp}). 
Furthermore $^*\RDL_{\wt\bG(s_0)}^{\wbG}(\chi) \widehat s_0^{-1}$ lies in a unipotent block.
Thanks to Theorem \ref{Conjunip} we can assume that $\bK\leq\bG(s)$.
The proof of \cite[23.4]{CabEnRedGp} can be applied and proves that $^*\RDL_{\wt\bG(s_0)}^{\wbG}(\chi) \widehat s_0^{-1}$ lies in the $d$-Harish-Chandra series of $(\bK,\kappa)$, hence 
$^*\RDL_{\wt\bG(s)}^{\wbG}(\chi) \widehat s^{-1} \in \cE(\wt\bG(s)^F, (\bK,\kappa))$. 
This proves $\Irr_0(\wt B)\subseteq \calG_{\bS}$.

Let $\overline \cP_{\wt B}:=\Upsilon^{-1}(\Irr_0(\wt B))$. We see that $\Upsilon'(\overline \cP_{\wt B})\subseteq \Irr(\wt b)$, since $\bl(\chi)=\bl(\Omega^\circ(\chi))^\wG$ for every $\chi\in\calG_\bS$ and $\wt b$ is the sum of all blocks $c$ of $\wt N$ such that $c^\wG$ is in $\wt B$. 

It remains to prove that $\Irr(\wt b)\subseteq \calN_\bS$. 
The blocks of $\wt b$ cover blocks $c$ of $\bL^F$. Hence, via Jordan decomposition, every height zero character of $c$ corresponds to a character in the $d$-Harish-Chandra series of $(\bK,\kappa)$ with $\bS=\Zent(\bK')_{\Phi_d}$. This implies $\Irr_0(\wt b)\subseteq \calN_\bS$. 

Accordingly the restriction of $\wt \Omega_\bS$ to $\Irr_0(\wt B)$ yields an injective map $$\wt \Omega_{\wt B}\colon \Irr_0(\wt B)\to\Irr(\wt b)$$ with the required equivariance properties.

In order to finish our proof it remains to show that $\wt \Omega_{\wt B} (\Irr_0(\wt B))=\Irr_0(\wt b)$.
Recall that $\chi_0=\Upsilon^\circ(s_0,\kappa,\eta)\in \Irr(B)$. Lemma \ref{lem_deg} shows that the $\ell$-part of the degree of $\Upsilon^\circ(s_0,\kappa,\rm 1_{W_{\wbG(s_0)^F}(\bK,\kappa)})$ is minimal amongst $\Irr(\wt B)$. Accordingly this character has height $0$.
We see that the block has defect $|\wt\bG(s)^F:\bK^F| \frac{|\bK^F|}{\kappa(1)_\ell}$. According to the description of the defect group this proves that $|\wt\bG(s)^F:\bK^F|_\ell=|W_{\wt\bG(s)^F}(\bK_{},\kappa)|_\ell$.
This proves that $\Irr_0(\wt B)=\Upsilon^{\circ}(\calP_{\wt B})$, where $\calP_{\wt B}$ is the set of triples $(s',\kappa ',\eta ')\in \cP_{\bS}$ with the following properties: 
\begin{itemize}
	\item $(s')_\lp \in \Zent((\wt\bG^*)^F)s_0$, $\kappa'=\kappa$, and 
	\item $\eta'\in\Irr_{\ell '}(W_{\wbG(s')}(\bK,\kappa)^F)$, and
	\item $\ell \nmid |W_{\wt\bG(s_0)}(\bK,\kappa)^F: W_{\wt\bG(s')}(\bK,\kappa)^F|$.
\end{itemize}
Note that $W_{\wt\bG(s')}(\bK,\kappa)^F=W_{\wt\bG(s')}(\bK)^F$ by the description of unipotent $d$-cuspidal pairs given in \cite[21.6]{CabEnRedGp} and \cite[proof of 3.3]{BMM}.

On the other hand considering the description of the defect group from 
 \cite[Lemma~4.16]{CabEngAdv} and the degrees given in Lemma \ref{lem_deg_loc} we see $\Irr_0(\wt b)=\Upsilon'^{\circ}(\calP_{\wt B})$. This implies that $\wt \Omega_{\wt B}$ defines a bijection with the required properties. 
\end{proof}

For later we point out the following equality. 
\begin{cor}\label{cor74}	\label{cor74b}
Let $B\in\Bl(\SL_n(\epsilon q))$ be covered by $\wt B_0$ in the situation of \Cref{thm72}.
Then $|\Irr_0(B)|=|\Irr_0(b)|$, where $N:=\wt N\cap \SL_n(\epsilon q)$ and $b\in\Bl(N)$ with $b^{\GF}=B$.
\end{cor}
\begin{proof}
	
	The map $\wt \Omega_{\wt B}$ is $\Irr(\wt G/G)$-compatible. Accordingly the number of constituents of $\Res^{\wG}_G(\chi)$ and $\Res^{\wN}_N(\Omega(\chi))$ coincide for all $\chi\in\Irr_0(\wt B)$. Also every character of $\Irr_0(B)$ and $\Irr_0(b)$ is the constituent of the the restriction of some character of $\Irr_0(\wt B)$ and $\Irr_0(\wt b)$, respectively, since $\ell\nmid |\wt G:G|$. Furthermore we see that all blocks of $G$ covered by $\wt B$ have the same number of height zero characters. Accordingly \[
	\bigcup_{\chi\in\Irr_ 0(\wt B)} \Irr(\Res_G^{\wt G}(\chi))=\bigcup_{B\in\Bl(G\mid \wt B)}\Irr_0(B).\] 
	This implies the stated equality via analogous local considerations. 
\end{proof}

\section{Proofs of the main Statements}
In this section we combine the results of the previous sections to provide a proof of \Cref{thmA} and \Cref{AMTypeA}. We are first concerned with the consequences towards the Alperin-McKay Conjecture and the inductive AM condition. At the end we turn our focus towards the blockwise Alperin weight conjecture and the blockwise Alperin weight condition.

\begin{proof}[Proof of Theorem \ref{thm11a}]
Let $\ell$ be a prime $\ell\nmid q(q-\epsilon)$, $B_0$ be an $\ell$-block of $\SL_n(\epsilon q)$ and $B$ the sum of $\GL_n(\epsilon q)$-conjugates of $B_0$. Assume that $\Out(G)_B$ is abelian. 
We prove the inductive AM condition as given in \cite[7.2]{AMSp} for $B$ via an application of 
Theorem \ref{IndAMCond}. While clearly Assumption (i) holds, Assumption (iii) follows from \ref{StarCondSL}. The bijection from Corollary \ref{thm72} has the properties required in \ref{thm24ii}. 
Assumption 2.4(iv) is known from Theorem \ref{thm12b}. 
Recall that $\Out(G)_B$ is assumed to be abelian. This implies that the inductive AM condition from \cite[7.2]{AMSp} holds for $B_0$.
\end{proof}

In the above we apply the criterion from Theorem \ref{NewIndAmCond} assuming \ref{NewIndAmCondv}. 
Corollary \ref{cor74} and Theorem \ref{thm12b} allow also an application of \cite[Theorem 4.1]{CabSpAMTypeA} towards the verification of the inductive AM condition for more $\ell$-blocks.
\begin{cor}\label{corindAM}
Let $\ell$ be a prime, $q$ a prime power and $\epsilon\in \{\pm 1\}$ with $\ell\nmid 3q( q-\epsilon)$, let $b$ be an $\ell$-block of $G:=\SL_n(\epsilon q)$. 
If for every $J$ with $G\leq J\leq \GL_n(\epsilon q)$ every $c\in \Bl(J\mid b)$ is $\GL(\epsilon q)$-stable then the inductive AM condition from Definition 7.2 of \cite{AMSp} holds for $b$.
\end{cor}
In \cite[\S 5]{CabSpAMTypeA} the condition on $b$ is studied using the $d$-cuspidal pair $(\bL,\zeta)$ associated with $b$. For every $J$ with $G\leq J\leq \GL_n(\epsilon q)$ every $c\in \Bl(J\mid b)$ is $\GL_n(\epsilon q)$-stable if $\Zent(\bL)$ is connected according to the considerations given there. 

\begin{proof}
The criterion \cite[Theorem 4.1]{CabSpAMTypeA} can also be applied to a $\wt G$-orbit of blocks of $\SL_n(\epsilon q)$. The assumptions 4.1(i)-(iv) are ensured as in the proof of \ref{thm11a}. The remaining condition (v) is satisfied by assumption. 
\end{proof}

The equality in Corollary \ref{cor74} is close to the Alperin-McKay conjecture for those blocks, just the block of the normalizer of the defect group is replaced by the induced block of the normalizer of the Sylow $\Phi_d$-torus. 
We use the result of \ref{prop75} to relate the blocks of those two groups.

\begin{proof}[Proof of Theorem \ref{thm13a}] Assume $G=\SL_n(\epsilon q)$ writes as $\bG^F$ as before. 
Let $B\in \Bl(\SL_n(\epsilon q))$, $N$ and $b\in \Bl(N)$ be corresponding to $B$ as in Corollary \ref{cor74b}. Then $|\Irr_0(B)|=|\Irr_0(b)|$ according to Corollary \ref{cor74b}. Denote by $(\bL ,\zeta)$ the $d$-cuspidal pair defining $B$ as in \cite[4.1]{CabEngAdv}. Note that $N=\NNN_G(\bL)$.

Some defect group $D$ of $B$ is also a defect group of any covering block of GL$_n(\epsilon q)$ and satisfies $\Cent_G(D)\leq \bL^F$ by Theorem \ref{71b} and \cite[\S 5.3]{CabEngAdv}.
Note that $D$ is also a defect group of $b$.
All characters of $L:=\bL^F$ extend to their stabilizers in $N$ by Theorem \ref{StarConddSpLevi}. For $\xi\in \Irr(L)$ the group $N_\xi/L$ satisfies the McKay Conjecture for $\ell$ by Proposition \ref{propMcKayrelWeyl}.
Now every block $\wt b$ of $\wt L:=(\Zent(\wt\bG)\bL)^F$
contains at least one character of $\Irr(\wt L\mid \zeta)$ and those are $d$-cuspidal according to \cite[proof of 1.10(i)]{CaEn94}. The block $\wt b$ is splendid Rickard equivalent to a unipotent block via a Jordan decomposition, see \cite{BonDatRou}. This block has a $d$-cuspidal unipotent character. 
 According to \cite[22.9]{CabEnRedGp} this block has central defect and is hence nilpotent. 
This implies that $\wt b$ is nilpotent, since the fusion systems of the blocks are preserved by splendid Rickard equivalence according to \cite[19.7]{Puig_book} and nilpotency of a block can be read off from its fusion system.
The $\ell$-group $D$ acts on the set $\Bl(\wt L\mid b)$ since $b$ is $D$-stable. Observe that $\ell\nmid |\Bl(\wt L\mid b)|$ since multiplication with characters of $\Irr(\wt L/L)$ defines an action on $\Bl(\wt L\mid b)$.

Accordingly there exists some block $\wt b\in \Bl(\wt L\mid b)$ that is $D$-stable. Since $\Bl(\wt L\mid b)$ forms an $\Irr(\wt L/L)$-orbit, all blocks in $\Bl(\wt L\mid b)$ are $D$-stable. Then the block $\wt c\in\Bl(\wt L D)$ covering $\wt b$ is nilpotent as well by \cite[Theorem 2]{Cab_p_ext_nilp}. Let $\wt b\in \Bl(LD\mid b)$ be covered by $ \wt c$. This block is inertial by \cite[Theorem 3.13]{Puig_nilpotent_extensions}. 
Hence Theorem \ref{prop75} can be applied and proves the statement.\end{proof}

While the above proves that all $\ell$-blocks of $\SL_n(\epsilon q)$ satisfy the Alperin-McKay conjecture for primes $\ell$ with $\ell\nmid 3q(q-\epsilon)$, we finally deduce the consequences of our considerations towards the Alperin weight Conjecture. 

\begin{proof}[Proof of Theorem \ref{thm11b}]
Let $B_0\in\Irr(\SL_n(\epsilon q))$ with defect group $D$, and $(\bL,\zeta)$ and $N$ be defined as in Corollary \ref{cor74}. Let $c\in 
\Bl(N)$ be the block with $c^{\SL_n(\epsilon q)}$ corresponding to $B_0$ and $c_0=\bl(\zeta)\in \Bl(\bL^F)$. By assumption $\Cent_\bG(D)$ is $d$-split and therefore it coincides with $\bL$, since $\bL$ is the minimal $d$-split Levi containing $\Cent^\circ_\bG(D)$ according to \cite[4.4(iii)]{CabEngAdv}.
The defect group hence satisfies $D= \Zent(\bL)^F_{\ell}$. This implies $N=\NNN_\GF(D)$. 

 The block $c_0:=\bl(\zeta)$ contains a character of $\cE_{\ell' }(\bL^F)$ and the characters of $\cE_{\ell' }(\bL^F)$ are trivial on $\Zent(L)_{\ell}$ according to \cite[1.2(v)]{CabEngAdv}.
On the other hand $c_0$ as block with central defect has a unique character that is trivial on $D$ and that corresponds to a defect zero character of $L/D$. 
For $\BB'_0:=\Irr(c_0)\cap \cE_\lp(\bL^F)$ this implies $\BB'_0= \Irr(c_0)\cap \Irr(L/D)$ and hence $\Irr(N\mid \BB_0')=\Irr(c)\cap \Irr(N/D)$.
As $c$ has normal defect, \[|\IBr(c)|=|\Irr(c)\cap \Irr(N/D)|=|\Irr(N\mid \BB'_0)|\].

Let $\wGF=\GL_n(\epsilon q)$ be defined as in as in \ref{71b}, $\wt L:=\Cent_{\wGF}(\Zent(\bL))$ and $\wt N:=\NNN_{\wGF}(\bL)$. 
We denote by $\wt B$ the sum of blocks of $\wGF$ covering $B_0$ and by $\wt b$ the sum of corresponding blocks of $\wt N$. 
Denote by $\wt b_0$ the sum of blocks of $\wt L$ covered by one of $\wt b$. 
The construction of $\wt \Omega_{\wt B}:\Irr_0(\wt B)\longrightarrow \Irr_0(\wt b)$ in \ref{thm72} implies 
\[\wt \Omega_{\wt B} (\wt {\BB})=\Irr(\wt N\mid \wt \BB'),\]
where $\wt {\BB}:= \Irr_0(\wt B)\cap \cE_{\ell'} (\wGF) $ and $\wt \BB':=\Irr(\wt b_0) \cap \cE_\lp(\wt L)$. 

Let $B$ be the sum of blocks in $\Bl(\GF)$ covered by a block of $\wt B$ and 
$b$ be the sum of blocks in $\Bl(N)$ covered by a blocks of $\wt b$. 
By the construction in Section \ref{RefomIndAM} we obtain a bijection $\Omega_B:\Irr_0(B)\longrightarrow \Irr_0(b)$. For $\BB:=\bigcup_{\chi\in \wt {\mathbb B}} \Irr(\Res^{\wG}_{G}(\chi))$ 
and $\BB':=\bigcup_{\psi\in \wt {\mathbb B'}} \Irr(\Res^{\wt L}_{L}(\psi))$ 
Corollary \ref{cor_AWC} implies $\Omega_B(\mathbb B)=\Irr(b)\cap \Irr(N/D)$. 

Note that $\mathbb B\subseteq \cE_{\ell'}(\GF)$. Since the $\ell$-modular decomposition matrix of $B_0$ is unitriangular with respect to the $\Aut(\GF)_B$-stable set $\BB_0$, see \cite[Theorem~C and Proposition~2.6(iii)]{GeckBS2}, this implies the inductive blockwise Alperin weight condition, see \cite[Theorem~1.2]{KosSpAMBAWCy}.
\end{proof}

We conclude by proving the Alperin weight conjecture for blocks with abelian defect in our situation. 
\begin{proof}[Proof of \ref{thm13b}]
Let $B$ be the sum of a $\wG$-orbit in $\Bl(\SL_n(\epsilon q))$.
Let $\bG:=\SL_n(\overline \FF_q)$ and $F:\bG\rightarrow\bG$ be a Frobenius endomorphism such that $\bG^F=\SL_n(\epsilon q)$. 
Let $(\bL,\zeta)$ be the $d$-cuspidal pair of a block $B_0$ from the blocks in $B$, $L:=\bL^F$ and $N:=\NNN_{\GF}(\bL)$. By \Cref{71b}, it follows that $\NNN_\GF(\bL)\geq \NNN_\GF(D)$ for some defect group $D$ of $B_0$ and 
hence by Brauer correspondence \cite[4.12]{NavBl} the blocks in $B$ correspond to a $\wt N$-orbit in $\Bl(N)$ of the same length. We denote the sum of those blocks by $b$.

Let $\Omega_B:\Irr_0(B)\rightarrow \Irr_0(b)$ be a bijection that is derived from the bijection of \Cref{thm72} as in Section 2.
The arguments in the proof of Theorem~\ref{thm11b} imply that 
\[\Omega_B(\Irr_0(B)\cap \cE_{\ell'}(\GF))=\Irr_0(b)\cap \Irr(N\mid \cE_\lp(\bL^F)).\]
Let $c_0:=\bl(\zeta)\in \Bl(\bL^F)$ and denote by $c$ the block of $N$ covering $c_0$. Let $r$ be the length of the $\wG$-orbit of $G$ containing $B_0$. Recall that $\cE_{\lp}(\GF)$ is a basic set of $\GF$ by \cite[Theorem C and Proposition 2.6(iii)]{GeckBS2}.
Then $b$ has $r$ summands, one is $c$, and
\[|\IBr(B_0)|= \frac{|\Irr_0(B)\cap \cE_{\ell'}(\GF)|}{r}=
\frac{|\Irr_0(b)\cap \Irr(N\mid \cE_{\ell'}(\bL^F)|}{r}=
|\Irr_0(c)\cap \Irr(N\mid \BB')|,\]
where $\BB':= \Irr(c_0)\cap \cE_\lp(\bL^F)$.
Then by \cite{BonDatRou} $c_0$ is basic Morita equivalent to a block $c_0'$ of $\Cent_{\bL^*}(s)^F$ above a unipotent block $c_0''\in\Bl(\Cent^\circ_{\bL^*}(s)^F)$ with central defect. The unipotent character of $c_0''$ is trivial on the central defect, hence restricts to an irreducible Brauer character and hence forms a basic set. By Clifford theory the unipotent characters of $c_0'$ form a basic set with a diagonal $\ell$-modular decomposition matrix as well, since $\Cent_{\bL^*}(s)^F/\Cent_{\bL^*}^\circ(s)^F$ is an $\ell'$-group. The Morita equivalence from \cite{BonDatRou} maps the unipotent characters to $\BB':=\Irr(c_0)\cap \cE_\lp(\bL^F)$ and hence $\Irr(c_0)\cap \cE_\lp(\bL^F)$ is a $N_{c_0}$-stable basic set with a unitriangular $\ell$-decomposition matrix. 
Recall that by the proof of Theorem \ref{thm13a} the assumptions of Proposition \ref{prop75} are satisfied. Then the assumption \ref{prop29b} applies and we obtain $|\IBr(c)|=|\Irr(N\mid \BB')|$.

Since $c$ satisfies the Alperin weight Conjecture by Proposition \ref{prop29a}, this implies the Alperin weight conjecture for $B_0$ and all blocks of $B$. 
\end{proof}

\end{document}